\documentclass[12pt]{article}
\usepackage[]{amsmath,amssymb}
\usepackage{amscd}
\usepackage{latexsym}
\usepackage{cite}
\usepackage{amsthm}

\newtheorem{definition}{Definition}[section]
\newtheorem{theorem}[definition]{Theorem}
\newtheorem{lemma}[definition]{Lemma}

\newtheorem{remark}[definition]{Remark}

\newtheorem{problem}[definition]{Problem}
\newtheorem{note}[definition]{Note}

\newtheorem{proposition}[definition]{Proposition}

\begin{document}
\title{\bf 
The compact presentation for the\\
alternating central extension \\
 of the $q$-Onsager algebra}
\author{
Paul Terwilliger 
}
\date{}

\maketitle
\begin{abstract}
 The $q$-Onsager algebra $O_q$ is defined by two generators and  two relations, called the $q$-Dolan/Grady relations.
 We investigate the alternating central extension $\mathcal O_q$ of $O_q$. The algebra $\mathcal O_q$ was introduced by 
Baseilhac and Koizumi, who called it the current algebra of $O_q$. Recently
 Baseilhac and Shigechi gave a presentation of $\mathcal O_q$ by generators and relations. 
 The presentation is attractive, but the multitude of generators and relations makes
the presentation unwieldy. In this paper we obtain a presentation of $\mathcal O_q$
 that involves a subset of the original set of generators and a very  manageable set of relations. We call this presentation the
 compact presentation of $\mathcal O_q$. This presentation resembles the compact presentation of
 the alternating central extension for the positive part of $U_q(\widehat{\mathfrak{sl}}_2)$.
\bigskip


 \end{abstract}
 
 \section{Introduction}
 \noindent We will be discussing the $q$-Onsager algebra $O_q$  \cite{bas1, qSerre}.
 This infinite-dimensional associative algebra is defined by two generators $W_0$, $W_1$ and  two relations, called the $q$-Dolan/Grady relations:
\begin{align*}
&\lbrack W_0, \lbrack W_0, \lbrack W_0, W_1\rbrack_q \rbrack_{q^{-1}} \rbrack =(q^2 - q^{-2})^2 \lbrack W_1, W_0 \rbrack,
\\
&\lbrack W_1, \lbrack W_1, \lbrack W_1, W_0\rbrack_q \rbrack_{q^{-1}}\rbrack = (q^2-q^{-2})^2 \lbrack W_0, W_1 \rbrack.
\end{align*}

\noindent
The algebra $O_q$ first appeared in algebraic combinatorics, in the study of association schemes that
are $P$-polynomial and $Q$-polynomial;
to our knowledge the $q$-Dolan/Grady relations first appeared in \cite[Lemma~5.4]{tersub3}. Discussions of the history can be found in
\cite{ito, qSerre, TD00}.
The algebra $O_q$ is the ``most general'' example of a tridiagonal algebra
\cite[Definition~3.9]{qSerre}.
In the topic of linear algebra, there is an object called a tridiagonal pair
\cite[Definition~1.1]{TD00}.
A finite-dimensional irreducible $O_q$-module is essentially the same thing as a tridiagonal pair of $q$-Racah type \cite[Theorem~3.10]{qSerre}. 
These tridiagonal pairs are classified up to isomorphism in \cite[Theorem~3.3]{ItoTer}. For more information about tridiagonal pairs, see \cite{ItoTerAug}, \cite[Section~19]{tbTD} and the references therein.
\medskip

\noindent  Over time the algebra $O_q$ found many applications in mathematics and physics; some examples are given below.
 The algebra  $O_q$ is used 
 to study boundary 
integrable systems 
\cite{
bas2,
bas1,
bas8,
basXXZ,
basBel,
BK05,
bas4,
basKoi,
basnc}.
The algebra $O_q$  can be realized as a left or right coideal subalgebra of the quantized enveloping algebra
$U_q(\widehat{\mathfrak{sl}}_2)$; see \cite{bas8, basXXZ, kolb}. The algebra $O_q$ is the simplest example of a quantum symmetric pair coideal subalgebra 
of affine type \cite[Example~7.6]{kolb}.
 In \cite{BK},
 two automorphisms of $O_q$ are introduced that resemble the Lusztig automorphisms of $U_q(\widehat{\mathfrak{sl}}_2)$; see also 
\cite{lusztigaut, z2z2z2}. These automorphisms are used in \cite{BK} to obtain a PBW basis for $O_q$, and they are used in \cite{diagram} to describe the Bockting double lowering operator
\cite{bockting, bocktingQexp} of a tridiagonal pair.
A Drinfeld type presentation of $O_q$ is obtained in \cite{LuWang}, and this  is used in \cite{LRW} to realize $O_q$ as an $\iota$Hall algebra of the projective line.
There is an algebra  homomorphism from  $O_q$ into the universal Askey-Wilson algebra 
\cite[Sections~9,10]{uaw}, \cite{pbw}. 
In \cite[Section~4]{basXXZ} some infinite-dimensional $O_q$-modules are constructed using $q$-vertex operators. 
\medskip

\noindent
 In \cite{BK05} Baseilhac and Koizumi introduce a current algebra  for $O_q$, in order to solve boundary integrable systems with hidden symmetries. We will denote this current algebra by $\mathcal O_q$.
In \cite[Definition~3.1]{basnc} Baseilhac and Shigechi give a presentation of $\mathcal O_q$ by generators and relations. 
By \cite[Theorem~9.14]{pbwqO} $\mathcal O_q$ is isomorphic to  $O_q\otimes  \mathbb F \lbrack z_1, z_2, \ldots \rbrack$,
where $\mathbb F$ is the ground field and $\lbrace z_n \rbrace_{n=1}^\infty$ are mutually commuting indeterminates.
Thus $\mathcal O_q$ is related to $O_q$ in roughly the same way that $\mathcal U^+_q$ is related to $U^+_q$,
where $U^+_q$ is the positive part of  $U_q(\widehat{\mathfrak{sl}}_2)$ \cite{damiani}
and $\mathcal U^+_q$ is 
 the alternating central extension of $U^+_q$ \cite{altCE}.
Motivated by this and as explained in \cite{pbwqO}, 
we call $\mathcal O_q$  the alternating central extension of $O_q$.  
\medskip

\noindent  The following recent results and conjectures are motivated by a  rough analogy between the pair $U^+_q$, $\mathcal U^+_q$ and the pair $O_q$, $\mathcal O_q$. 
In \cite{alternating}
we introduce the alternating elements in $U^+_q$, and show how they are related to the PBW basis for $U^+_q$ due to Damiani \cite{damiani}.
In \cite{conj} we conjecture the existence of some elements in $O_q$ that are similarly related to the  PBW basis for $O_q$ given in \cite{BK}.
In \cite{altCE} we show that for $\mathcal U^+_q$, the generators in the defining presentation give a PBW basis.
In \cite{pbwqO} we establish the corresponding result for $\mathcal O_q$.
In \cite[Theorem~3.1]{FMA} the algebra $\mathcal U^+_q$ is shown to have a  Freidel-Maillet presentation. The corresponding result for $\mathcal O_q$ is that it admits a presentation as a reflection algebra, and this is established in
 \cite[Theorem~2]{basnc}.
The algebra $U^+_q$ is the positive part of $U_q(\widehat{\mathfrak{sl}}_2)$, so of course there exists an injective algebra homomorphism
 $U^+_q\to U_q(\widehat{\mathfrak{sl}}_2)$. In \cite[Proposition~5.18]{FMA} Baseilhac gives a similar injective   algebra 
homomorphism $\mathcal U^+_q \to U_q(\widehat{\mathfrak{gl}}_2)$.
We mentioned earlier that
$O_q$  can be realized as a coideal subalgebra of 
$U_q(\widehat{\mathfrak{sl}}_2)$. We expect that $\mathcal O_q$ can be realized as a coideal
subalgebra of
$U_q(\widehat{\mathfrak{gl}}_2)$ in a similar way, and we will try to establish this in the future. 
\medskip




\noindent In  \cite[Definition~3.1]{compactUqp} we introduced the compact presentation of $\mathcal U^+_q$. In \cite[Section~7]{compactUqp} 
we described some features of $\mathcal U^+_q$ that are illuminated by the compact presentation.
In the present paper we do something similar for $\mathcal O_q$.
\medskip

\noindent Our results are summarized as follows.
The presentation of $\mathcal O_q$ from \cite[Definition~3.1]{basnc} is
reproduced in Definition \ref{def:Aq}
below.
The presentation is attractive, but the
 multitude of generators and relations makes the presentation unwieldy. In our first main result Theorem \ref{thm:m1Int}, we obtain a presentation of $\mathcal O_q$
 that involves a  subset of the original set of generators and a very manageable set of relations. We call this presentation the compact presentation of $\mathcal O_q$. 
 In our second main result  Theorem \ref{thm:m2}, we factor the vector space $\mathcal O_q$ into a tensor product of two subalgebras.
 \begin{theorem} \label{thm:m1Int}
 The algebra $\mathcal O_q$  has a presentation by generators 
$\mathcal W_0$, $\mathcal W_1$, $\lbrace \mathcal {\tilde G}_{k+1} \rbrace_{k \in \mathbb N}$ and relations
\begin{enumerate}
\item[\rm (i)]
$\lbrack \mathcal W_0, \lbrack \mathcal W_0, \lbrack \mathcal W_0, \mathcal W_1\rbrack_q \rbrack_{q^{-1}} \rbrack =(q^2-q^{-2})^2 \lbrack \mathcal W_1, \mathcal W_0 \rbrack$;
\item[\rm (ii)]
$\lbrack \mathcal W_1, \lbrack \mathcal W_1, \lbrack \mathcal W_1, \mathcal W_0\rbrack_q \rbrack_{q^{-1}}\rbrack = (q^2-q^{-2})^2  \lbrack \mathcal W_0, \mathcal W_1 \rbrack$;
\item[\rm (iii)] 
$\lbrack \mathcal W_0, \mathcal {\tilde G}_1 \rbrack = 
\lbrack \mathcal W_0, \lbrack \mathcal W_0, \mathcal W_1 \rbrack_q \rbrack$;
\item[\rm (iv)] $\lbrack \mathcal {\tilde G}_1, \mathcal W_1 \rbrack = 
\lbrack \lbrack \mathcal W_0, \mathcal W_1 \rbrack_q, \mathcal W_1 \rbrack$;
\item[\rm (v)] 
for $k\geq 1$,
\begin{align*}
&\lbrack \mathcal {\tilde G}_{k+1}, \mathcal W_0 \rbrack = 
\frac{
\lbrack \mathcal W_0, \lbrack \mathcal W_0, \lbrack \mathcal W_1,
\mathcal {\tilde G}_k
\rbrack_q
\rbrack_q
\rbrack}{(q^2-q^{-2})^2};
\end{align*}
\item[\rm (vi)] for $k\geq 1$,
\begin{align*}
\lbrack \mathcal W_1, \mathcal {\tilde G}_{k+1}\rbrack = 
\frac{
\lbrack\lbrack \lbrack \mathcal {\tilde G}_k, \mathcal W_0 \rbrack_q, 
\mathcal W_1 \rbrack_q,
\mathcal W_1
\rbrack}{(q^2-q^{-2})^2};
\end{align*}
\item[\rm (vii)] for $k, \ell \in \mathbb N$,
\begin{align*}
\lbrack \mathcal {\tilde G}_{k+1}, \mathcal {\tilde G}_{\ell+1} \rbrack=0.
\end{align*}
\end{enumerate}
\end{theorem}

\noindent The $\mathcal O_q$-generators in Theorem \ref{thm:m1Int} are called essential.
Before stating our second main result, we make a few definitions.
Let $\langle \mathcal W_0, \mathcal W_1 \rangle $ denote the subalgebra of $\mathcal O_q$ generated by $\mathcal W_0$, $\mathcal W_1$.
Let $\mathcal {\tilde G}$ denote the subalgebra of $\mathcal O_q$ generated by $\lbrace \mathcal {\tilde G}_{k+1} \rbrace_{k \in \mathbb N}$.
 Item (i) below appeared in \cite[Theorem~10.3]{pbwqO} and is included for completeness.

\begin{theorem}\label{thm:m2}
For the  algebra $\mathcal O_q$ the following {\rm (i)--(iv)} hold:
\begin{enumerate}
\item[\rm (i)] there exists an algebra isomorphism $O_q \to \langle \mathcal W_0, \mathcal W_1 \rangle $ that sends $W_0 \mapsto \mathcal W_0$ and
$W_1 \mapsto \mathcal W_1$;
\item[\rm (ii)] there exists an algebra isomorphism  $\mathbb F \lbrack z_1, z_2, \ldots \rbrack \to \mathcal {\tilde G}$ that sends $z_n \mapsto \mathcal {\tilde G}_n$ for $n\geq 1$;
\item[\rm (iii)] the multiplication map
\begin{align*}
\langle  \mathcal W_0, \mathcal W_1\rangle \otimes \mathcal {\tilde G} &\to
	       \mathcal O_q 
	       \\
w \otimes g &\mapsto      wg            
\end{align*}
is an isomorphism of vector spaces;
\item[\rm (iv)] the multiplication map
\begin{align*}
  \mathcal {\tilde G} \otimes
\langle  \mathcal W_0, \mathcal W_1\rangle
 &\to
	       \mathcal O_q 
	       \\
g \otimes w &\mapsto      gw           
\end{align*}
is an isomorphism of vector spaces.
\end{enumerate}
\end{theorem}
\medskip

\noindent The paper is organized as follows. Section 2 contains some preliminaries. In Section 3, we recall the algebra $O_q$ and describe its basic properties.
In Section 4, we describe the algebra $\mathcal O_q$ and its relationship  to $O_q$.
In Section 5, we introduce the essential generators for $\mathcal O_q$.
In Section 6, we first identify some relations in $\mathcal O_q$ that are satisfied by the essential generators. We then define an algebra $\mathcal O^\vee_q$ by generators and relations,
using the essential generators and the relations we identified. In Sections 7, 8 we describe a filtration of $\mathcal O_q$ from two points of view.
In Section 9 we describe a filtration of $\mathcal O^\vee_q$. In Section 10,  the above
filtrations  are used to show that $\mathcal O_q$, $\mathcal O^\vee_q$ are isomorphic, and  this fact is used to
prove Theorems \ref{thm:m1Int}, \ref{thm:m2}.
Section 11 contains some comments about $\mathcal O_q$ that are motivated by Theorem \ref{thm:m2}(iii),(iv).
In Section 12, we use an automorphism of $\mathcal O_q$ to obtain a variation on Theorems \ref{thm:m1Int}, \ref{thm:m2} and the results of Section 11.
In Section 13, we give some suggestions for future research.

 \section{Preliminaries}
 We now begin our formal argument.
Throughout  the paper, the following notational conventions are in effect.
Recall the natural numbers 
$\mathbb N =\lbrace 0,1,2,\ldots \rbrace$ and integers $\mathbb Z = \lbrace 0, \pm 1, \pm 2, \ldots \rbrace$.
 Let $\mathbb F$ denote a field. Every vector space and tensor product
 mentioned in this paper is over $\mathbb F$. Every algebra mentioned in this paper is associative, over $\mathbb F$, and 
has a multiplicative identity. 
 Let $\mathcal A$ denote an algebra. By an {\it automorphism} of $\mathcal A$
we mean an algebra isomorphism $\mathcal A\rightarrow \mathcal A$. The algebra $\mathcal A^{\rm opp}$ consists of the vector space $\mathcal A$ and the multiplication map $\mathcal A \times \mathcal A \rightarrow \mathcal A$, $(a,b)\to ba$.
By an {\it antiautomorphism} of $\mathcal A$ we mean an algebra isomorphism $\mathcal A \rightarrow \mathcal A^{\rm opp}$.

 \begin{definition}\rm 
(See \cite[p.~299]{damiani}.)
Let $ \mathcal A$ denote an algebra. A {\it Poincar\'e-Birkhoff-Witt} (or {\it PBW}) basis for $\mathcal A$
consists of a subset $\Omega \subseteq \mathcal A$ and a linear order $<$ on $\Omega$
such that the following is a basis for the vector space $\mathcal A$:
\begin{align*}
a_1 a_2 \cdots a_n \qquad n \in \mathbb N, \qquad a_1, a_2, \ldots, a_n \in \Omega, \qquad
a_1 \leq a_2 \leq \cdots \leq a_n.
\end{align*}
We interpret the empty product as the multiplicative identity in $\mathcal A$.
\end{definition}
 \begin{definition}\label{def:gr}\rm 
 A  {\it grading} of an algebra $\mathcal A$ is a sequence  $\lbrace \mathcal A_n \rbrace_{n \in \mathbb N}$ of subspaces of $\mathcal A$ such that
(i) $1 \in \mathcal A_0$; (ii) the sum $\mathcal A = \sum_{n \in \mathbb N} \mathcal A_n$ is direct; (iii) $\mathcal A_r \mathcal A_s \subseteq \mathcal A_{r+s} $ for $r,s\in \mathbb N$.
\end{definition}
\begin{definition} \label{def:filtration} \rm (See \cite[p.~202]{carter}.)
A {\it filtration} of an algebra $\mathcal A$ is a sequence  $\lbrace \mathcal A_n \rbrace_{n \in \mathbb N}$ of subspaces of $\mathcal A$ such that
(i) $1 \in \mathcal A_0$; (ii) $\mathcal A_{n-1} \subseteq \mathcal A_n$ for $n\geq 1$; (iii) $\mathcal A = \cup_{n \in \mathbb N} \mathcal A_n$;
(iv) $\mathcal A_r \mathcal A_s \subseteq \mathcal A_{r+s} $ for $r,s\in \mathbb N$.
\end{definition}
\begin{definition}\label{def:poly}\rm
Let $\lbrace z_n \rbrace_{n=1}^\infty$ denote mutually commuting indeterminates. Let $\mathbb F \lbrack z_1, z_2, \ldots \rbrack$ denote
the algebra consisting of the polynomials in $z_1, z_2, \ldots $ that have all coefficients in $\mathbb F$.
For notational convenience define $z_0=1$.
\end{definition}

 \noindent Fix a nonzero $q \in \mathbb F$
that is not a root of unity.
Recall the notation
\begin{eqnarray*}
\lbrack n\rbrack_q = \frac{q^n-q^{-n}}{q-q^{-1}}
\qquad \qquad n \in \mathbb N.
\end{eqnarray*}

\section{The $q$-Onsager algebra $O_q$}
In this section we recall the $q$-Onsager algebra $ O_q$.
\noindent For elements $X, Y$ in any algebra, define their
commutator and $q$-commutator by 
\begin{align*}
\lbrack X, Y \rbrack = XY-YX, \qquad \qquad
\lbrack X, Y \rbrack_q = q XY- q^{-1}YX.
\end{align*}
\noindent Note that 
\begin{align*}
\lbrack X, \lbrack X, \lbrack X, Y\rbrack_q \rbrack_{q^{-1}} \rbrack
= 
X^3Y-\lbrack 3\rbrack_q X^2YX+ 
\lbrack 3\rbrack_q XYX^2 -YX^3.
\end{align*}

\begin{definition} \label{def:U} \rm
(See \cite[Section~2]{bas1}, \cite[Definition~3.9]{qSerre}.)
Define the algebra $O_q$ by generators $W_0$, $W_1$ and relations
\begin{align}
\label{eq:qOns1}
&\lbrack W_0, \lbrack W_0, \lbrack W_0, W_1\rbrack_q \rbrack_{q^{-1}} \rbrack =(q^2 - q^{-2})^2 \lbrack W_1, W_0 \rbrack,
\\
\label{eq:qOns2}
&\lbrack W_1, \lbrack W_1, \lbrack W_1, W_0\rbrack_q \rbrack_{q^{-1}}\rbrack = (q^2-q^{-2})^2 \lbrack W_0, W_1 \rbrack.
\end{align}
We call $O_q$ the {\it $q$-Onsager algebra}.
The relations \eqref{eq:qOns1}, \eqref{eq:qOns2}  are called the {\it $q$-Dolan/Grady relations}.
\end{definition}
\begin{remark}\rm In \cite{BK} Baseilhac and Kolb define the $q$-Onsager algebra in a slightly more general way that involves two scalar parameters $c, q$. Our $O_q$ is their
$q$-Onsager algebra with $c=q^{-1}(q-q^{-1})^2$.
\end{remark}
\noindent We mention some symmetries of $O_q$. 

\begin{lemma}
\label{lem:aut} There exists an automorphism $\sigma$ of $O_q$ that sends $W_0 \leftrightarrow W_1$.
Moreover $\sigma^2 = 1$.
\end{lemma}

\begin{lemma}\label{lem:antiaut} {\rm (See \cite[Lemma~2.5]{z2z2z2}.)}
There exists an antiautomorphism $\dagger$ of $O_q$ that fixes each of $W_0$, $W_1$.
 Moreover $\dagger^2=1$.
\end{lemma}

\begin{lemma} The maps $\sigma$, $\dagger$ commute.
\end{lemma}
\begin{proof} This is readily checked.
\end{proof}

\begin{definition}\label{def:tau} \rm Let $\tau$ denote the composition of $\sigma$ and $\dagger$. Note that $\tau$ is an antiautomorphism of $O_q$ that sends
$W_0 \leftrightarrow W_1$. We have $\tau^2 = 1$.
\end{definition}

  \section{The algebra $\mathcal O_q$}
  
  \noindent Recall from Section 1 that $\mathcal O_q$ is the alternating central extension of $O_q$. In this section we formally define $\mathcal O_q$, and describe some basic properties.
  
\begin{definition}\rm
\label{def:Aq}
(See 
\cite{BK05}, \cite[Definition~3.1]{basnc}.)
Define the algebra $\mathcal O_q$
by generators
\begin{align}
\label{eq:4gens}
\lbrace \mathcal W_{-k}\rbrace_{n\in \mathbb N}, \qquad  \lbrace \mathcal  W_{k+1}\rbrace_{n\in \mathbb N},\qquad  
 \lbrace \mathcal G_{k+1}\rbrace_{n\in \mathbb N},
\qquad
\lbrace \mathcal {\tilde G}_{k+1}\rbrace_{n\in \mathbb N}
\end{align}
 and the following relations. For $k, \ell \in \mathbb N$,
\begin{align}
&
 \lbrack \mathcal W_0, \mathcal W_{k+1}\rbrack= 
\lbrack \mathcal W_{-k}, \mathcal W_{1}\rbrack=
({\mathcal{\tilde G}}_{k+1} - \mathcal G_{k+1})/(q+q^{-1}),
\label{eq:3p1}
\\
&
\lbrack \mathcal W_0, \mathcal G_{k+1}\rbrack_q= 
\lbrack {\mathcal{\tilde G}}_{k+1}, \mathcal W_{0}\rbrack_q= 
\rho  \mathcal W_{-k-1}-\rho 
 \mathcal W_{k+1},
\label{eq:3p2}
\\
&
\lbrack \mathcal G_{k+1}, \mathcal W_{1}\rbrack_q= 
\lbrack \mathcal W_{1}, {\mathcal {\tilde G}}_{k+1}\rbrack_q= 
\rho  \mathcal W_{k+2}-\rho 
 \mathcal W_{-k},
\label{eq:3p3}
\\
&
\lbrack \mathcal W_{-k}, \mathcal W_{-\ell}\rbrack=0,  \qquad 
\lbrack \mathcal W_{k+1}, \mathcal W_{\ell+1}\rbrack= 0,
\label{eq:3p4}
\\
&
\lbrack \mathcal W_{-k}, \mathcal W_{\ell+1}\rbrack+
\lbrack \mathcal W_{k+1}, \mathcal W_{-\ell}\rbrack= 0,
\label{eq:3p5}
\\
&
\lbrack \mathcal W_{-k}, \mathcal G_{\ell+1}\rbrack+
\lbrack \mathcal G_{k+1}, \mathcal W_{-\ell}\rbrack= 0,
\label{eq:3p6}
\\
&
\lbrack \mathcal W_{-k}, {\mathcal {\tilde G}}_{\ell+1}\rbrack+
\lbrack {\mathcal {\tilde G}}_{k+1}, \mathcal W_{-\ell}\rbrack= 0,
\label{eq:3p7}
\\
&
\lbrack \mathcal W_{k+1}, \mathcal G_{\ell+1}\rbrack+
\lbrack \mathcal  G_{k+1}, \mathcal W_{\ell+1}\rbrack= 0,
\label{eq:3p8}
\\
&
\lbrack \mathcal W_{k+1}, {\mathcal {\tilde G}}_{\ell+1}\rbrack+
\lbrack {\mathcal {\tilde G}}_{k+1}, \mathcal W_{\ell+1}\rbrack= 0,
\label{eq:3p9}
\\
&
\lbrack \mathcal G_{k+1}, \mathcal G_{\ell+1}\rbrack=0,
\qquad 
\lbrack {\mathcal {\tilde G}}_{k+1}, {\mathcal {\tilde G}}_{\ell+1}\rbrack= 0,
\label{eq:3p10}
\\
&
\lbrack {\mathcal {\tilde G}}_{k+1}, \mathcal G_{\ell+1}\rbrack+
\lbrack \mathcal G_{k+1}, {\mathcal {\tilde G}}_{\ell+1}\rbrack= 0.
\label{eq:3p11}
\end{align}
In the above equations $\rho = -(q^2-q^{-2})^2$. The generators 
\eqref{eq:4gens} are called {\it alternating}.
\noindent For notational convenience define
\begin{align}
{\mathcal G}_0 = -(q-q^{-1})\lbrack 2 \rbrack^2_q, \qquad \qquad 
{\mathcal {\tilde G}}_0 = -(q-q^{-1}) \lbrack 2 \rbrack^2_q.
\label{eq:GG0}
\end{align}
\end{definition}
\begin{note}\rm In 
\cite{BK05}, \cite{basnc}
the algebra $\mathcal O_q$ is denoted by $\mathcal A_q$ and called the current algebra for $O_q$.
In \cite{pbwqO} we proposed to call $\mathcal O_q$ the alternating central extension of $O_q$.
\end{note}

\begin{proposition} \label{lem:pbw} {\rm (See \cite[Theorem~6.1]{pbwqO}.)} A PBW basis for $\mathcal O_q$ is obtained by its alternating generators in any linear order $<$ such that
\begin{align}
\mathcal G_{i+1} < \mathcal W_{-j} < \mathcal W_{k+1} < \mathcal {\tilde G}_{\ell+1}\qquad \qquad (i,j,k, \ell \in \mathbb N).
\label{eq:order}
\end{align}
\end{proposition}
\noindent Next we describe some symmetries of $\mathcal O_q$.
\begin{lemma}
\label{lem:autA} {\rm (See \cite[Remark~1]{basBel}.)} There exists an automorphism $\sigma$ of $\mathcal O_q$ that sends
\begin{align*}
\mathcal W_{-k} \mapsto \mathcal W_{k+1}, \qquad
\mathcal W_{k+1} \mapsto \mathcal W_{-k}, \qquad
\mathcal G_{k+1} \mapsto \mathcal {\tilde G}_{k+1}, \qquad
\mathcal {\tilde G}_{k+1} \mapsto \mathcal G_{k+1}
\end{align*}
 for $k \in \mathbb N$. Moreover $\sigma^2 = 1$.
\end{lemma}

\begin{lemma}\label{lem:antiautA} {\rm (See \cite[Lemma~3.7]{z2z2z2}.)} There exists an antiautomorphism $\dagger$ of $\mathcal O_q$ that sends
\begin{align*}
\mathcal W_{-k} \mapsto \mathcal W_{-k}, \qquad
\mathcal W_{k+1} \mapsto \mathcal W_{k+1}, \qquad
\mathcal G_{k+1} \mapsto \mathcal {\tilde G}_{k+1}, \qquad
\mathcal {\tilde G}_{k+1} \mapsto \mathcal G_{k+1}
\end{align*}
for $k \in \mathbb N$. Moreover $\dagger^2=1$.
\end{lemma}

\begin{lemma} \label{lem:sdcomA} The maps $\sigma$, $\dagger $ commute.
\end{lemma}
\begin{proof} Routine.
\end{proof}
\begin{definition}\label{def:tauA}\rm Let $\tau$ denote the composition of the automorphism $\sigma$ from Lemma \ref{lem:autA} and the antiautomorphism $\dagger$ from Lemma \ref{lem:antiautA}. 
 Note that $\tau$ is an antiautomorphism of $\mathcal O_q$ that sends
\begin{align*}
\mathcal W_{-k} \mapsto \mathcal W_{k+1}, \qquad
\mathcal W_{k+1} \mapsto \mathcal W_{-k}, \qquad
\mathcal G_{k+1} \mapsto \mathcal G_{k+1}, \qquad
\mathcal {\tilde G}_{k+1} \mapsto \mathcal {\tilde G}_{k+1}
\end{align*}
\noindent for $k \in \mathbb N$. We have $\tau^2=1$.
\end{definition}

\noindent Next we discuss how $\mathcal O_q$ is related to $O_q$.

\begin{lemma}
\label{lem:newrels1}
For the algebra $\mathcal O_q$,
\begin{align}
\label{eq:r1}
&\lbrack \mathcal W_0, \mathcal {\tilde G}_1 \rbrack = 
\lbrack \mathcal W_0, \lbrack \mathcal W_0, \mathcal W_1 \rbrack_q \rbrack,
\\
\label{eq:r2}
&\lbrack \mathcal {\tilde G}_1, \mathcal W_1 \rbrack = 
\lbrack \lbrack \mathcal W_0, \mathcal W_1 \rbrack_q, \mathcal W_1 \rbrack.
\end{align}
\end{lemma}
\begin{proof} By \eqref{eq:3p1} with $k=0$,
\begin{align}
(q+ q^{-1}) \lbrack \mathcal W_0, \mathcal W_1 \rbrack = \mathcal {\tilde G}_1 - \mathcal G_1.
\label{eq:one}
\end{align}
By \eqref{eq:3p2} with $k=0$,
\begin{align}
 \lbrack \mathcal W_0, \mathcal G_1 \rbrack_q = \lbrack \mathcal {\tilde G}_1, \mathcal W_0 \rbrack_q.
\label{eq:two}
\end{align}
\noindent To get \eqref{eq:r1}, eliminate $\mathcal G_1$ from \eqref{eq:two} using \eqref{eq:one}, and simplify the
result. To get \eqref{eq:r2}, apply the antiautomorphism $\tau$ to each side of \eqref{eq:r1}.
\end{proof}

\noindent The following result appeared in  \cite[line (3.7)]{basBel}. We give a short proof for the sake of completeness.

  \begin{lemma} \label{eq:DG} {\rm (See \cite[line (3.7)]{basBel}.)}
For the algebra $\mathcal O_q$,
\begin{align}
\label{eq:qOns1p}
&\lbrack \mathcal W_0, \lbrack \mathcal W_0, \lbrack \mathcal W_0, \mathcal W_1\rbrack_q \rbrack_{q^{-1}} \rbrack =(q^2-q^{-2})^2 \lbrack \mathcal W_1, \mathcal W_0 \rbrack,
\\
\label{eq:qOns2p}
&\lbrack \mathcal W_1, \lbrack \mathcal W_1, \lbrack \mathcal W_1, \mathcal W_0\rbrack_q \rbrack_{q^{-1}}\rbrack = (q^2-q^{-2})^2  \lbrack \mathcal W_0, \mathcal W_1 \rbrack.
\end{align}
\end{lemma}
\begin{proof} We first obtain \eqref{eq:qOns1p}.  By \eqref{eq:3p4} with $(k,\ell)=(1,0)$ we obtain $\lbrack \mathcal W_{-1}, \mathcal W_{0}\rbrack = 0$. By \eqref{eq:3p2} with $k=0$ we obtain
\begin{align*}
\rho \mathcal W_{-1} = \rho \mathcal W_1+\lbrack \mathcal {\tilde G}_1, \mathcal W_0 \rbrack_q.
\end{align*}
Using these facts and  \eqref{eq:r1}, we obtain
\begin{align*}
0 &= \rho \lbrack \mathcal W_0, \mathcal W_{-1}\rbrack\\
&= \lbrack \mathcal W_0, \rho \mathcal W_1 + \lbrack \mathcal {\tilde G}_1, \mathcal W_0 \rbrack_q \rbrack
\\ 
&= \rho \lbrack \mathcal W_0,  \mathcal W_1\rbrack + \lbrack \lbrack \mathcal W_0, \mathcal {\tilde G}_1\rbrack, \mathcal W_0 \rbrack_q
\\
&= \rho \lbrack \mathcal W_0,  \mathcal W_1\rbrack + \lbrack \lbrack \mathcal W_0,\lbrack \mathcal W_0, \mathcal W_1\rbrack_q \rbrack, \mathcal W_0  \rbrack_q
\\
&= \rho \lbrack \mathcal W_0,  \mathcal W_1\rbrack - \lbrack  \mathcal W_0, \lbrack \mathcal W_0,\lbrack \mathcal W_0, \mathcal W_1 \rbrack_q \rbrack \rbrack_{q^{-1}}
\\
&= \rho \lbrack \mathcal W_0,  \mathcal W_1\rbrack - \lbrack  \mathcal W_0, \lbrack \mathcal W_0, \lbrack \mathcal W_0, \mathcal W_1 \rbrack_q \rbrack_{q^{-1}} \rbrack
\end{align*}
\noindent which implies \eqref{eq:qOns1p}. To get \eqref{eq:qOns2p}, apply the automorphism $\sigma$ to each side of \eqref{eq:qOns1p}.
\end{proof}

\begin{lemma}
\label{lem:iota}  {\rm (See \cite[Theorem~10.3]{pbwqO}.)} 
There exists an algebra homomorphism $\iota: O_q \to \mathcal O_q$ that sends $W_0 \mapsto \mathcal W_0$ and
$W_1 \mapsto \mathcal W_1$. Moreover, $\iota$ is injective.
\end{lemma}

\begin{lemma} 
\label{lem:diag}
The following diagrams commute:

\begin{equation*}
{\begin{CD}
O_q @>\iota  >> \mathcal O_q
              \\
         @V \sigma VV                   @VV \sigma V \\
         O_q @>>\iota >
                                 \mathcal O_q
                        \end{CD}}  	
         \qquad \qquad                
    {\begin{CD}
O_q @>\iota  >> \mathcal O_q
              \\
         @V \dagger VV                   @VV \dagger V \\
         O_q @>>\iota >
                                 \mathcal O_q
                        \end{CD}}  	     \qquad \qquad                
   {\begin{CD}
O_q @>\iota  >> \mathcal O_q
              \\
         @V \tau VV                   @VV \tau V \\
         O_q @>>\iota >
                                 \mathcal O_q
                        \end{CD}}  	                                       		    
\end{equation*}
\end{lemma}
\begin{proof} Chase the $O_q$-generators $W_0$, $W_1$ around the diagrams, using Lemmas \ref{lem:aut}, \ref{lem:antiaut} and Definition \ref{def:tau} along with
 Lemmas \ref{lem:autA}, \ref{lem:antiautA} and Definition \ref{def:tauA}.
\end{proof}

\section{The essential generators for $\mathcal O_q$}

\noindent We return our attention to the presentation of $\mathcal O_q$ in Definition \ref{def:Aq}.
There is a certain redundancy among the alternating generators \eqref{eq:4gens}. 
We are going to express all of them in terms of $\mathcal W_0$, $\mathcal W_1$, $\lbrace \mathcal {\tilde G}_{k+1} \rbrace_{k\in \mathbb N}$. Following \cite[Section~4]{z2z2z2}
we use
\eqref{eq:3p2}, \eqref{eq:3p3} to recursively obtain
$\mathcal W_{-k}$, $\mathcal W_{k+1}$ for $k\geq 1$:
\begin{align*}
\mathcal W_{-1} &= \mathcal W_1 -\frac{\lbrack \mathcal { \tilde G}_{1}, 
\mathcal W_0\rbrack_q}{(q^2-q^{-2})^2},
\\
\mathcal W_{3} &=  \mathcal W_1 
-
\frac{\lbrack \mathcal {\tilde G}_{1}, 
\mathcal W_0\rbrack_q}{(q^2-q^{-2})^2}
-
\frac{\lbrack
 \mathcal W_1,
\mathcal {\tilde G}_{2} 
\rbrack_q}{(q^2-q^{-2})^2},
\\
\mathcal W_{-3} &=  \mathcal W_1 
-
\frac{\lbrack \mathcal {\tilde G}_{1}, 
\mathcal W_0\rbrack_q}{(q^2-q^{-2})^2}
-
\frac{\lbrack
\mathcal W_1,
\mathcal {\tilde G}_{2} 
\rbrack_q}{(q^2-q^{-2})^2}
-
\frac{\lbrack \mathcal {\tilde G}_{3}, 
\mathcal W_0\rbrack_q}{(q^2-q^{-2})^2},
\\
\mathcal W_{5} &= \mathcal W_1 
-
\frac{\lbrack \mathcal {\tilde G}_{1}, 
\mathcal W_0\rbrack_q}{(q^2-q^{-2})^2}
-
\frac{\lbrack
\mathcal W_1,
\mathcal {\tilde G}_{2} 
\rbrack_q}{(q^2-q^{-2})^2}
-
\frac{\lbrack \mathcal {\tilde G}_{3}, 
\mathcal W_0\rbrack_q}{(q^2-q^{-2})^2}
-
\frac{\lbrack
\mathcal W_1, \mathcal {\tilde G}_{4} 
\rbrack_q}{(q^2-q^{-2})^2},
\\
\mathcal W_{-5} &= \mathcal W_1 
-
\frac{\lbrack\mathcal {\tilde G}_{1}, \mathcal W_0\rbrack_q}{(q^2-q^{-2})^2}
-
\frac{\lbrack
\mathcal W_1,
\mathcal {\tilde G}_{2} 
\rbrack_q}{(q^2-q^{-2})^2}
-
\frac{\lbrack \mathcal  {\tilde G}_{3}, 
 \mathcal W_0\rbrack_q}{(q^2-q^{-2})^2}
-
\frac{\lbrack
\mathcal W_1,
 \mathcal {\tilde G}_{4} 
\rbrack_q}{(q^2-q^{-2})^2}
-
\frac{\lbrack  \mathcal {\tilde G}_{5}, 
\mathcal W_0\rbrack_q}{(q^2-q^{-2})^2},
\\
& \cdots
\end{align*}
\begin{align*}
\mathcal W_2 &=  \mathcal W_0 -
\frac{
\lbrack \mathcal W_1, \mathcal {\tilde G}_1\rbrack_q}{(q^2-q^{-2})^2
},
\\
\mathcal W_{-2} &=  \mathcal W_0 -
\frac{
\lbrack \mathcal W_1, \mathcal {\tilde G}_1\rbrack_q}{(q^2-q^{-2})^2}
-
\frac{
\lbrack \mathcal {\tilde G}_2, W_0\rbrack_q}{(q^2-q^{-2})^2},
\\
\mathcal W_{4} &=  \mathcal W_0 -
\frac{
\lbrack \mathcal W_1, \mathcal {\tilde G}_1\rbrack_q}{(q^2-q^{-2})^2}
-
\frac{
\lbrack \mathcal {\tilde G}_2,  \mathcal W_0\rbrack_q}{(q^2-q^{-2})^2}
-
\frac{
\lbrack \mathcal W_1, \mathcal {\tilde G}_3\rbrack_q}{(q^2-q^{-2})^2},
\\
\mathcal W_{-4} &=  \mathcal W_0 -
\frac{
\lbrack \mathcal W_1, \mathcal {\tilde G}_1\rbrack_q}{(q^2-q^{-2})^2}
-
\frac{
\lbrack \mathcal {\tilde G}_2, \mathcal W_0\rbrack_q}{(q^2-q^{-2})^2}
-
\frac{
\lbrack \mathcal W_1, \mathcal {\tilde G}_3\rbrack_q}{(q^2-q^{-2})^2}
-
\frac{
\lbrack \mathcal {\tilde G}_4, \mathcal W_0\rbrack_q}{(q^2-q^{-2})^2},
\\
 \mathcal W_{6} &=  \mathcal W_0 -
\frac{
\lbrack \mathcal W_1, \mathcal {\tilde G}_1\rbrack_q}{(q^2-q^{-2})^2}
-
\frac{
\lbrack \mathcal {\tilde G}_2,  \mathcal W_0\rbrack_q}{(q^2-q^{-2})^2}
-
\frac{
\lbrack  \mathcal W_1, \mathcal {\tilde G}_3\rbrack_q}{(q^2-q^{-2})^2}
-
\frac{
\lbrack \mathcal {\tilde G}_4, \mathcal W_0\rbrack_q}{(q^2-q^{-2})^2}
-
\frac{
\lbrack \mathcal W_1, \mathcal {\tilde G}_5\rbrack_q}{(q^2-q^{-2})^2},
\\
&\cdots
\end{align*}

\noindent 
The recursion shows that for  $k \geq 1$, 
the generators $ \mathcal W_{-k}$, $ \mathcal W_{k+1}$
are given as follows. For odd $k=2r+1$,
\begin{align}
\mathcal W_{-k} = \mathcal W_1
- 
\sum_{\ell=0}^r 
\frac{
\lbrack 
 \mathcal {\tilde G}_{2\ell+1},  \mathcal W_0\rbrack_q
}{(q^2-q^{-2})^2} 
-
\sum_{\ell=1}^r 
\frac{
\lbrack
 \mathcal W_1,
 \mathcal {\tilde G}_{2\ell} 
 \rbrack_q
}{(q^2-q^{-2})^2},
\label{eq:WmkA}
\\
\mathcal W_{k+1} =  \mathcal W_0
- 
\sum_{\ell=0}^r 
\frac{
\lbrack 
\mathcal W_1,
\mathcal {\tilde G}_{2\ell+1}\rbrack_q
}{(q^2-q^{-2})^2} 
-
\sum_{\ell=1}^r 
\frac{
\lbrack
 \mathcal {\tilde G}_{2\ell},
 \mathcal W_0
 \rbrack_q
}{(q^2-q^{-2})^2}.
\label{eq:Wkp1A}
\end{align}
For even $k=2r$,
\begin{align}
&\mathcal W_{-k} = \mathcal W_0
- 
\sum_{\ell=0}^{r-1} 
\frac{
\lbrack  \mathcal W_1,
 \mathcal {\tilde G}_{2\ell+1}\rbrack_q
}{(q^2-q^{-2})^2} 
-
\sum_{\ell=1}^r 
\frac{
\lbrack
  \mathcal {\tilde G}_{2\ell},
 \mathcal W_0
 \rbrack_q
}{(q^2-q^{-2})^2},
\label{eq:WmkB}
\\
\label{eq:Wkp1B}
 &\mathcal W_{k+1} =  W_1
- 
\sum_{\ell=0}^{r-1} 
\frac{
\lbrack 
\mathcal {\tilde G}_{2\ell+1}, \mathcal W_0\rbrack_q
}{(q^2-q^{-2})^2} 
-
\sum_{\ell=1}^r 
\frac{
\lbrack
 \mathcal W_1,
 \mathcal {\tilde G}_{2\ell} 
 \rbrack_q
}{(q^2-q^{-2})^2}. 
\end{align}
Next we use 
\eqref{eq:3p1} to
obtain the generators
$\lbrace \mathcal G_{k+1}\rbrace_{k \in \mathbb N}$:
\begin{align}
\mathcal G_{k+1} =  \mathcal {\tilde G}_{k+1} + (q+q^{-1}) 
\lbrack \mathcal W_1,  \mathcal W_{-k} \rbrack
\qquad \qquad (k \in \mathbb N).
\label{eq:getrid}
\end{align}
We have expressed the alternating generators \eqref{eq:4gens} in terms of
$\mathcal W_0$, $\mathcal W_1$, $\lbrace \mathcal {\tilde G}_{k+1} \rbrace_{k \in \mathbb N}$. We record the consequence.
\begin{lemma} 
\label{lem:minGen}
The algebra $\mathcal O_q$ is generated by $\mathcal W_0$, $\mathcal W_1$, $\lbrace \mathcal {\tilde G}_{k+1} \rbrace_{k \in \mathbb N}$.
  \end{lemma}
  \begin{definition}
  \label{def:ess}
  \rm The $\mathcal O_q$-generators  from Lemma \ref{lem:minGen} will be called {\it essential}.
  \end{definition}

  \section{The algebra $\mathcal O^\vee_q$}
  
  \noindent In Definition \ref{def:ess} we defined the essential generators for $\mathcal O_q$. 
  Our next goal is to obtain a presentation of $\mathcal O_q$ by generators and relations, where the generators are the essential ones.   Let us consider what the relations should be. By \eqref{eq:3p10} we have
  $\lbrack \mathcal {\tilde G}_{k+1}, \mathcal {\tilde G}_{\ell+1} \rbrack=0$ for $k, \ell \in \mathbb N$. We have the relations in Lemmas \ref{lem:newrels1},  \eqref{eq:DG}.
  In the next result we list some additional relations satisfied by the essential generators.

\begin{lemma}
\label{lem:newrels}
For the algebra $\mathcal O_q$ the following relations hold for  $k\geq 1$:
\begin{align}
\label{eq:r3}
&\lbrack \mathcal {\tilde G}_{k+1}, \mathcal W_0 \rbrack = 
\frac{
\lbrack \mathcal W_0, \lbrack \mathcal W_0, \lbrack \mathcal W_1,
\mathcal {\tilde G}_k
\rbrack_q
\rbrack_q
\rbrack}{(q^2-q^{-2})^2},
\\
\label{eq:r4}
&\lbrack \mathcal W_1, \mathcal {\tilde G}_{k+1}\rbrack = 
\frac{
\lbrack\lbrack \lbrack \mathcal {\tilde G}_k, \mathcal W_0 \rbrack_q, 
\mathcal W_1 \rbrack_q,
\mathcal W_1
\rbrack}{(q^2-q^{-2})^2}.
\end{align}
\end{lemma}
\begin{proof} We first obtain  \eqref{eq:r3}. Eliminate $\mathcal G_{k+1}$ in \eqref{eq:3p2} using \eqref{eq:3p1}  to obtain        
\begin{align}
\lbrack \mathcal W_0, {\mathcal {\tilde G}}_{k+1} \rbrack=
\lbrack \mathcal W_0, \lbrack \mathcal W_0, \mathcal W_{k+1}\rbrack_q \rbrack.
\label{eq:elim3}
\end{align}
Now use in order
\eqref{eq:elim3},
\eqref{eq:3p4},
\eqref{eq:3p3}  
to obtain
\begin{align*}
\lbrack \mathcal W_0, \mathcal {\tilde G}_{k+1} \rbrack 
&= 
 \lbrack \mathcal W_0, \lbrack \mathcal W_0,  \mathcal W_{k+1}\rbrack_q \rbrack
\\
&= 
\lbrack \mathcal W_0, \lbrack \mathcal W_0,  \mathcal W_{k+1}\rbrack \rbrack_q
\\
&=
\lbrack \mathcal W_0, \lbrack \mathcal W_0,  
\mathcal W_{k+1}-\mathcal W_{1-k} \rbrack \rbrack_q
\\
&= 
\rho^{-1}\lbrack \mathcal W_0, \lbrack \mathcal W_0, \lbrack \mathcal W_1,
\mathcal {\tilde G}_k
\rbrack_q
\rbrack
\rbrack_q
\\
&= \rho^{-1}
\lbrack \mathcal W_0, \lbrack \mathcal W_0, \lbrack \mathcal W_1,
\mathcal {\tilde G}_k
\rbrack_q
\rbrack_q
\rbrack
\end{align*}
which implies \eqref{eq:r3}. To obtain  \eqref{eq:r4}, apply the antiautomorphism $\tau$ to each side of \eqref{eq:r3}.
\end{proof}
\noindent We just identified some relations that are satisfied by the essential generators of $\mathcal O_q$. We now define an algebra $\mathcal O^\vee_q$ by
generators and relations, using the essential generators and the relations we identified.

\begin{definition} \rm
\label{thm:m1} Define the algebra $\mathcal O^\vee_q$  by generators 
$\mathcal W_0$, $\mathcal W_1$, $\lbrace \mathcal {\tilde G}_{k+1} \rbrace_{k \in \mathbb N}$ and relations
\begin{enumerate}
\item[\rm (i)]
$\lbrack \mathcal W_0, \lbrack \mathcal W_0, \lbrack \mathcal W_0, \mathcal W_1\rbrack_q \rbrack_{q^{-1}} \rbrack =(q^2-q^{-2})^2 \lbrack \mathcal W_1, \mathcal W_0 \rbrack$;
\item[\rm (ii)]
$\lbrack \mathcal W_1, \lbrack \mathcal W_1, \lbrack \mathcal W_1, \mathcal W_0\rbrack_q \rbrack_{q^{-1}}\rbrack = (q^2-q^{-2})^2  \lbrack \mathcal W_0, \mathcal W_1 \rbrack$;
\item[\rm (iii)] 
$\lbrack \mathcal W_0, \mathcal {\tilde G}_1 \rbrack = 
\lbrack \mathcal W_0, \lbrack \mathcal W_0, \mathcal W_1 \rbrack_q \rbrack$;
\item[\rm (iv)] $\lbrack \mathcal {\tilde G}_1, \mathcal W_1 \rbrack = 
\lbrack \lbrack \mathcal W_0, \mathcal W_1 \rbrack_q, \mathcal W_1 \rbrack$;
\item[\rm (v)] 
for $k\geq 1$,
\begin{align*}
&\lbrack \mathcal {\tilde G}_{k+1}, \mathcal W_0 \rbrack = 
\frac{
\lbrack \mathcal W_0, \lbrack \mathcal W_0, \lbrack \mathcal W_1,
\mathcal {\tilde G}_k
\rbrack_q
\rbrack_q
\rbrack}{(q^2-q^{-2})^2};
\end{align*}
\item[\rm (vi)] for $k\geq 1$,
\begin{align*}
\lbrack \mathcal W_1, \mathcal {\tilde G}_{k+1}\rbrack = 
\frac{
\lbrack\lbrack \lbrack \mathcal {\tilde G}_k, \mathcal W_0 \rbrack_q, 
\mathcal W_1 \rbrack_q,
\mathcal W_1
\rbrack}{(q^2-q^{-2})^2};
\end{align*}
\item[\rm (vii)] for $k, \ell \in \mathbb N$,
\begin{align*}
\lbrack \mathcal {\tilde G}_{k+1}, \mathcal {\tilde G}_{\ell+1} \rbrack=0.
\end{align*}
\end{enumerate}
The $\mathcal O^\vee_q$-generators $\mathcal W_0$, $\mathcal W_1$, $\lbrace \mathcal {\tilde G}_{k+1} \rbrace_{k \in \mathbb N}$
are called {\it essential}.
\end{definition}
\begin{remark}\rm 
Referring to Definition \ref{thm:m1}, the relation (iii) (resp. (iv)) is obtained from relation (v) (resp. (vi)) by setting $k=0$ and using \eqref{eq:GG0}.
\end{remark}

\begin{lemma}
\label{lem:step1} There exists an algebra homomorphism $\natural : \mathcal O^\vee_q \to \mathcal O_q$ that sends
\begin{align*}
\mathcal W_0 \mapsto \mathcal W_0, \qquad
\mathcal W_1 \mapsto \mathcal W_1, \qquad
\mathcal {\tilde G}_{k+1} \mapsto \mathcal {\tilde G}_{k+1}, \qquad k \in \mathbb N.
\end{align*}
\noindent Moreover, $\natural $ is surjective.
  \end{lemma}
  \begin{proof} The algebra homomorphism $\natural $ exists because the essential $\mathcal O_q$-generators satisfy the defining relations for $\mathcal O^\vee_q$. The map $\natural$
  is surjective by Lemma \ref{lem:minGen}.
  \end{proof}

 \noindent We are going to prove that  $\natural$ is bijective.
The proof will be completed in Theorem \ref{thm:nat}.
  In the next three sections we describe a filtration of $\mathcal O_q$ and a filtration of $\mathcal O^\vee_q$ that will help us complete the proof.
  \medskip

 \medskip

\section{A filtration of $\mathcal O_q$}

\noindent Recall the filtration concept from Definition \ref{def:filtration}.
In this section, we describe a filtration of $\mathcal O_q$ that will help us prove that the map $\natural$ from Lemma \ref{lem:step1}  is bijective.

\begin{definition}\label{def:degreeO}
\rm
To each alternating  $\mathcal O_q$-generator we assign a degree, as follows. For $k \in \mathbb N$,
\bigskip
 
\centerline{
\begin{tabular}[t]{c|cccc}
  $u$ & $\mathcal G_{k+1}$ & $\mathcal W_{-k}$ & $\mathcal W_{k+1}$ & $\mathcal {\tilde G}_{k+1}$ 
   \\
\hline
${\rm deg}(u)$ & $2k+2$ & $2k+1$ & $2k+1$ & $2k+2$
\end{tabular}}
\bigskip

\end{definition}

\begin{definition}\label{def:AnO} \rm For notational convenience, abbreviate $\mathcal O = \mathcal O_q$. 
For $d \in \mathbb N$ let $\mathcal O_d$ denote the subspace of $\mathcal O$ spanned by the products $a_1a_2\cdots a_n$ of alternating generators
such that $n\in \mathbb N$ and $\sum_{i=1}^n {\rm deg}(a_i) \leq d$.
\end{definition}
\noindent Using Definition \ref{def:AnO} we routinely obtain
(i) $\mathcal O_0 = \mathbb F1$;
(ii) $\mathcal O_{d-1} \subseteq \mathcal O_d$ for $d\geq 1$;
(iii) $\mathcal O = \cup_{d \in \mathbb N} \mathcal O_d$;
(iv) $\mathcal O_r \mathcal O_s \subseteq \mathcal O_{r+s}$ for $r,s\in \mathbb N$.
By these comments and Definition \ref{def:filtration}  the sequence $\lbrace \mathcal O_d\rbrace_{d \in \mathbb N}$ is a filtration of $\mathcal O_q$. 

\begin{note}\label{rem:BO} \rm
In \cite[Section~7]{pbwqO} the filtration $\lbrace \mathcal O_d\rbrace_{d \in \mathbb N}$  is denoted $\lbrace B_d\rbrace_{d \in \mathbb N}$.
\end{note}
 
\begin{lemma} \label{lem:factor}
For the algebra $\mathcal O_q$ the following {\rm (i), (ii)} hold:
\begin{enumerate}
\item[\rm (i)] 
the subspace 
 $\mathcal O_1$ has basis $1, \mathcal W_0, \mathcal W_1$;
\item[\rm (ii)] for $d\geq 2$ we have
\begin{align}
\label{eq:odd}
\mathcal O_{d} =E_d+ \mathcal O_{d-1} + \sum_{k=1}^{d-1} {\rm Span}(\mathcal O_k \mathcal O_{d-k}),
\end{align}
\noindent where $E_d=0$ if $d$ is odd and $E_d=\mathbb F \mathcal {\tilde G}_n$ if $d=2n$ is even.
\end{enumerate}
\end{lemma}
\begin{proof}(i) The elements  $\mathcal W_0, \mathcal W_1$ are the unique alternating generators of degree 1. The elements $1, \mathcal W_0, \mathcal W_1$ are linearly independent by
Proposition  \ref{lem:pbw}.
\\
\noindent (ii) Let $\mathcal O'_d$ denote the vector space on the right in \eqref{eq:odd}. We show that $\mathcal O_d = \mathcal O'_d$. We have $\mathcal O_d \supseteq \mathcal O'_d$ by 
Definition \ref{def:degreeO} and the comments below Definition \ref{def:AnO}.
Next we show that $\mathcal O_d \subseteq \mathcal O'_d$.
The vector space $\mathcal O_{d}$ is spanned by the  products $a_1a_2\cdots a_r$ of alternating generators
such that $r \in \mathbb N$ and $\sum_{i=1}^r {\rm deg}(a_i) \leq d$.
Let $w=a_1a_2\cdots a_r$ denote such a product. We show that $w \in \mathcal O'_d$.
We may assume that $\sum_{i=1}^r {\rm deg}(a_i) =d$;
otherwise $w\in \mathcal O_{d-1}\subseteq \mathcal O'_d$. Assume for the moment that $r\geq 2$. Write $w=w_1 w_2$ with $w_1=a_1$ and $w_2 = a_2\cdots a_r$. Let $k$ denote the degree of $a_1$.
By construction  $1 \leq k \leq d-1$.
Also by construction $w_1 \in \mathcal O_k$ and $w_2 \in \mathcal O_{d-k}$.  By these comments $w=w_1 w_2 \in \mathcal O_k \mathcal O_{d-k} \subseteq \mathcal O'_d$.
Next assume that $r=1$, so $w=a_1$ and ${\rm deg}(a_1)=d$. Suppose that $d=2n+1$ is odd. Then $w=\mathcal W_{-n}$ or $w=\mathcal W_{n+1}$. In either case
$w \in \mathcal O_{d-1}+ {\rm Span}(\mathcal O_1 \mathcal O_{d-1}) + {\rm Span}(\mathcal O_{d-1} \mathcal O_1)\subseteq \mathcal O'_d$ in 
view of \eqref{eq:WmkA}--\eqref{eq:Wkp1B}.
Next suppose that $d=2n$ is even. Then $w=\mathcal G_n$ or $w=\mathcal {\tilde G}_n$. In the first case
 $w \in E_d + {\rm Span}(\mathcal O_1 \mathcal O_{d-1}) + {\rm Span}(\mathcal O_{d-1}\mathcal O_1) \subseteq \mathcal O'_d$ in view of \eqref{eq:getrid}. In the second case $w \in E_d \subseteq \mathcal O'_d$.
We have shown that $\mathcal O_d \subseteq \mathcal O'_d$. By the above comments $\mathcal O_d=\mathcal O'_d$.
\end{proof}

\section{A filtration of $\mathcal O_q$, cont.}
\noindent In the previous section we described a filtration of $\mathcal O_q$. In this section we consider the filtration from another point of view.
\medskip

\noindent 
In \cite[Theorem~9.14]{pbwqO} we gave an algebra isomorphism $\phi: O_q \otimes \mathbb F \lbrack z_1, z_2, \ldots \rbrack \to \mathcal O_q$. The isomorphism is described as follows.
 In \cite[Proposition~8.12, Lemma~8.16]{pbwqO}
we displayed some elements $\lbrace \mathcal Z_n \rbrace_{n=1}^\infty$ in the center of $\mathcal O_q$.
 The map $\phi$ sends
\begin{align}
W_0 \otimes 1 \mapsto \mathcal W_0, \qquad \quad
W_1 \otimes 1 \mapsto \mathcal W_1, \qquad \quad 
1 \otimes z_n \mapsto \mathcal Z_n, \qquad n\geq 1.
\label{eq:phiAction}
\end{align}

\noindent Recall the filtration $\lbrace \mathcal O_d \rbrace_{d \in \mathbb N}$ of $\mathcal O_q$ from below Definition \ref{def:AnO}.
Our next goal is to describe $\phi^{-1}(\mathcal O_d)$ for $d \in \mathbb N$.
In this description, we will use 
a filtration of $O_q$ and a grading of $\mathbb F \lbrack z_1, z_2, \ldots \rbrack$. 
\medskip

\noindent We recall from \cite[Section~4]{pospart}
a filtration of $O_q$. For notational convenience abbreviate $O=O_q$.
For $d \in \mathbb N$ let  $O_d$  denote the subspace of $O$
spanned by the products
$g_1g_2\cdots g_n$ such that $0 \leq n \leq d$ and
$g_i$ is among $W_0, W_1$ for $1 \leq i \leq n$.
We routinely obtain
 (i) $O_0 = \mathbb F 1$; 
(ii) 
 $O_{d-1} \subseteq O_d$ for $d\geq 1$;
 (iii) $O = \cup_{d \in \mathbb N} O_d$;
 (iv)  $O_r O_s \subseteq  O_{r+s}$
for $r,s \in \mathbb N$. 
By these comments and Definition \ref{def:filtration} the sequence
$\lbrace O_d \rbrace_{d \in \mathbb N}$ 
is a filtration of $O_q$.
\medskip

\noindent Next consider the algebra $\mathbb F \lbrack z_1, z_2, \ldots \rbrack$. 
First we describe a basis for the vector space $\mathbb F \lbrack z_1, z_2, \ldots \rbrack$.
For $n \in \mathbb N$, a {\it partition of $n$} is a sequence $\lambda = \lbrace \lambda_i \rbrace_{i=1}^\infty$
of natural numbers such that $\lambda_i \geq \lambda_{i+1}$ for $i\geq 1$ and $n=\sum_{i=1}^\infty \lambda_i$.
Let the set $\Lambda_n$ consist of the partitions of $n$. Define $\Lambda = \cup_{n \in \mathbb N} \Lambda_n$. 
For $\lambda \in \Lambda$ define $z_\lambda = \prod_{i=1}^\infty z_{\lambda_i}$. The elements 
$\lbrace z_\lambda\rbrace_{\lambda \in \Lambda}$ form a basis for the vector space $\mathbb F \lbrack z_1, z_2, \ldots \rbrack$.
\medskip

\noindent
Next we construct a grading for the algebra $\mathbb F \lbrack z_1, z_2, \ldots \rbrack$.
For notational convenience abbreviate  $Z=\mathbb F \lbrack z_1, z_2, \ldots \rbrack$.
  For $n \in \mathbb N$ let $Z_n$ denote the subspace of
$Z$ with basis $\lbrace Z_\lambda \rbrace_{\lambda \in \Lambda_n}$. For example $ Z_0 = \mathbb F 1$.
The sum $Z = \sum_{n\in \mathbb N} Z_n$ is direct. Moreover $Z_r Z_s\subseteq Z_{r+s}$
for $r,s\in \mathbb N$. By these comments and Definition \ref{def:gr}
the subspaces $\lbrace Z_n \rbrace_{n\in \mathbb N}$ form
a grading of the algebra $Z$. 
\medskip

\begin{proposition}\label{def:cOn} 
For $d \in \mathbb N$,
\begin{align}
\label{eq:skip}
\phi^{-1}(\mathcal O_d) = \sum_{k=0}^{\lfloor d/2 \rfloor}  O_{d-2k} \otimes Z_{k}.
\end{align}
\end{proposition}
\begin{proof} By \cite[Definition~9.7]{pbwqO} and the proof of \cite[Theorem~9.14]{pbwqO}.
\end{proof}

\section{A filtration of  $\mathcal O^\vee_q$}

\noindent In the previous two sections, we described a filtration of $\mathcal O_q$ from several points of view.
In this section, we describe a filtration of $\mathcal O^\vee_q$.

\begin{definition}\label{def:degreeV}
\rm
To each essential $\mathcal O^\vee_q$-generator  we assign a degree, as follows: 
\bigskip
 
\centerline{
\begin{tabular}[t]{c|ccc}
  $u$ & $\mathcal W_0$ & $\mathcal W_1$ & $\mathcal {\tilde G}_{k+1}$ 
   \\
\hline
${\rm deg}(u)$ & $1$ & $1$ & $2k+2$
\end{tabular}}
\noindent In the above table $k \in \mathbb N$.
\bigskip

\end{definition}

\begin{definition}\label{def:AnV} \rm For notational convenience, abbreviate $\mathcal O^\vee = \mathcal O^\vee_q$. 
For $d \in \mathbb N$ let $\mathcal O^\vee_d$ denote the subspace of $\mathcal O^\vee$ spanned by the products $a_1a_2\cdots a_n$ of essential generators
such that $n \in \mathbb N$ and $\sum_{i=1}^n {\rm deg}(a_i) \leq d$. 
\end{definition}
 \noindent Using Definition \ref{def:AnV} we routinely obtain
 (i) $\mathcal O^\vee_0 = \mathbb F1$;
 (ii) $\mathcal O^\vee_{d-1} \subseteq \mathcal O^\vee_d$ for $d\geq 1$;
(iii) $\mathcal O^\vee = \cup_{d \in \mathbb N} \mathcal O^\vee_d$;
(iv) $\mathcal O^\vee_r \mathcal O^\vee_s \subseteq \mathcal O^\vee_{r+s}$ for $r,s\in \mathbb N$.
By these comments and Definition \ref{def:filtration}, the sequence $\lbrace \mathcal O^\vee_d\rbrace_{d \in \mathbb N}$ is a filtration of $\mathcal O^\vee_q$. 
\medskip

 \noindent The following description of $\lbrace \mathcal O^\vee_d \rbrace_{d \in \mathbb N}$ is reminiscent of Lemma \ref{lem:factor}.
 
\begin{lemma} \label{lem:factorV}
For the algebra $\mathcal O^\vee_q$ the following {\rm (i), (ii)} hold:
\begin{enumerate}
\item[\rm (i)] 
the subspace 
 $\mathcal O^\vee_1$ has basis $1, \mathcal W_0, \mathcal W_1$;
\item[\rm (ii)] for $d\geq 2$ we have
\begin{align}
\label{eq:oddV}
\mathcal O^\vee_{d} =E^\vee_d+ \mathcal O^\vee_{d-1} + \sum_{k=1}^{d-1} {\rm Span}(\mathcal O^\vee_k \mathcal O^\vee_{d-k}),
\end{align}
\noindent where $E^\vee_d=0$ if $d$ is odd and $E^\vee_d=\mathbb F \mathcal {\tilde G}_n$ if $d=2n$ is even.
\end{enumerate}
\end{lemma}
\begin{proof}(i) The elements  $\mathcal W_0, \mathcal W_1$ are the unique essential $\mathcal O^\vee_q$-generators of degree 1. The elements $1, \mathcal W_0, \mathcal W_1$ are linearly independent in $\mathcal O^\vee_q$, because
their $\natural$-images in $\mathcal O_q$ are linearly independent by Lemma \ref{lem:factor}(i).
\\
\noindent (ii) Let $\mathcal O^{\vee \prime}_d$ denote the vector space on the right in \eqref{eq:oddV}. We show that $\mathcal O^\vee_d = \mathcal O^{\vee \prime}_d$. We have $\mathcal O^\vee_d \supseteq \mathcal O^{\vee \prime}_d$ by 
Definition \ref{def:degreeV} and the comments below Definition \ref{def:AnV}.
Next we show that $\mathcal O^\vee_d \subseteq \mathcal O^{\vee \prime}_d$.
The vector space $\mathcal O^\vee_{d}$ is spanned by the  products $a_1a_2\cdots a_r$ of essential generators
such that $r \in \mathbb N$ and $\sum_{i=1}^r {\rm deg}(a_i) \leq d$.
Let $w=a_1a_2\cdots a_r$ denote such a product. We show that  $w \in \mathcal O^{\vee \prime}_d$.
We may assume that $\sum_{i=1}^r {\rm deg}(a_i) =d$;
otherwise $w\in \mathcal O^\vee_{d-1}\subseteq \mathcal O^{\vee \prime}_d$. Assume for the moment that $r\geq 2$. Write $w=w_1 w_2$ with $w_1=a_1$ and $w_2 = a_2\cdots a_r$. Let $k$ denote the degree of $a_1$.
By construction  $1 \leq k \leq d-1$.
Also by construction $w_1 \in \mathcal O^\vee_k$ and $w_2 \in \mathcal O^\vee_{d-k}$.  By these comments $w=w_1 w_2 \in \mathcal O^\vee_k \mathcal O^\vee_{d-k} \subseteq \mathcal O^{\vee \prime}_d$.
Next assume that $r=1$, so $w=a_1$ and ${\rm deg}(a_1)=d$. By assumption $d\geq 2$. The generators $\mathcal W_0, \mathcal W_1$ have degree 1,  so
 $d=2n$ is even and $w=\mathcal {\tilde G}_n$.  Consequently $w \in E^\vee_d \subseteq \mathcal O^{\vee \prime}_d$.
We have shown that $\mathcal O^\vee_d \subseteq \mathcal O^{\vee\prime}_d$. By the above comments $\mathcal O^\vee_d=\mathcal O^{\vee\prime}_d$.
\end{proof}

 \noindent The following description of $\lbrace \mathcal O^\vee_d \rbrace_{d \in \mathbb N}$ is reminiscent of Proposition \ref{def:cOn}.

\begin{proposition}\label{def:AnV2} 
For $d \in \mathbb N$ the subspace $\mathcal O^\vee_d$ is spanned by the products $a_1a_2\cdots a_n$ of essential generators
such that both
\begin{enumerate}
\item[\rm (i)]  $n \in \mathbb N$ and $\sum_{i=1}^n {\rm deg}(a_i) \leq d$;
\item[\rm (ii)]  ${\rm deg}(a_i) = 1$ implies ${\rm deg}(a_{i-1})=1$ for $2 \leq i\leq n$.
\end{enumerate}
\end{proposition}
\begin{proof} Consider a sequence  $(a_1,a_2,\ldots, a_n)$ of essential generators that satisfies (i). Call $n$ the length of the sequence, and note that $n\leq d$.
By an inversion for $(a_1,a_2, \ldots, a_n)$ we mean an ordered pair of integers $(i,j)$ such that  $1 \leq i<j\leq n$ and ${\rm deg}(a_i) \not=1$ and ${\rm deg}(a_j)=1$.
The sequence $(a_1,a_2,\ldots, a_n)$ has no inversions if and only if it satisfies (ii).
By Definition \ref{def:AnV}, $\mathcal O^\vee_d $ is  spanned by the products $a_1a_2\cdots a_n$ of essential generators that satisfy (i).
Let $U_d$ denote the span of the products $a_1a_2\cdots a_n$ of essential generators that satisfy (i), (ii).
So  $U_d \subseteq  \mathcal O^\vee_d $. We show that  $U_d =  \mathcal O^\vee_d $. To do this, we assume that $U_d \not=  \mathcal O^\vee_d $ and get a contradiction.
By the assumption, there exists a sequence $(a_1, a_2, \ldots, a_n)$ of essential generators that satisfies (i) and the product $a_1a_2\cdots a_n$ is not contained in $U_d$. Call such a sequence a counterexample (CE).
Let $N$ denote the maximal length of a CE.
Let $M$ denote the minimal number of inversions possessed by a CE of length $N$. Note that $M\geq 1$.
There exists a CE $(a_1,a_2, \ldots, a_N)$ that has exactly  $M$ inversions.
Since $(a_1,a_2,\ldots, a_N)$ is a CE, it does not satisfy (ii). So there exists an integer $i$ $(2 \leq i \leq N)$ such that ${\rm deg}(a_{i-1}) \not=1$ and ${\rm deg}(a_i)=1$.
Define the products $x=a_1 a_2 \cdots a_{i-2}$ and $y = a_{i+1} \cdots a_{N-1} a_N$. So $a_1a_2 \cdots a_N = x a_{i-1}a_i y$.
By construction exists $k \in \mathbb N$ such that  $a_{i-1} = \mathcal {\tilde G}_{k+1}$. Also by construction, $a_i =\mathcal W_\xi$ where $\xi \in \lbrace 0,1\rbrace$. 
We claim that $x \lbrack a_{i-1}, a_i\rbrack y\in U_d$. To prove the claim, we separate the cases $k\geq 1$ and $k=0$. Until further notice, assume that $k\geq 1$.
By the form of the relations (v), (vi) in 
Definition \ref{thm:m1}, the commutator $\lbrack a_{i-1}, a_i \rbrack$ is a linear combination of
\begin{align*}
& \mathcal {\tilde G}_{k} \mathcal W_{1-\xi} \mathcal W_{\xi} \mathcal W_{\xi}, \qquad
\mathcal W_{1-\xi} \mathcal {\tilde G}_{k} \mathcal W_{\xi} \mathcal W_{\xi}, \qquad
\mathcal W_{\xi} \mathcal {\tilde G}_{k} \mathcal W_{1-\xi} \mathcal W_{\xi}, 
\\
&\mathcal W_{\xi} \mathcal W_{1-\xi}  \mathcal {\tilde G}_{k} \mathcal W_{\xi}, \qquad
\mathcal W_{\xi} \mathcal W_{\xi}  \mathcal {\tilde G}_{k} \mathcal W_{1-\xi}, \qquad
\mathcal W_{\xi} \mathcal W_{\xi} \mathcal W_{1-\xi}   \mathcal {\tilde G}_{k}.
   \end{align*}
   Therefore $x\lbrack a_{i-1}, a_i\rbrack y$ is a linear combination of
   \begin{align}
    \label{eq:longer}
& x\mathcal {\tilde G}_{k} \mathcal W_{1-\xi} \mathcal W_{\xi} \mathcal W_{\xi}y, \qquad
x\mathcal W_{1-\xi} \mathcal {\tilde G}_{k} \mathcal W_{\xi} \mathcal W_{\xi}y, \qquad
x\mathcal W_{\xi} \mathcal {\tilde G}_{k} \mathcal W_{1-\xi} \mathcal W_{\xi}y, 
\\
&x\mathcal W_{\xi} \mathcal W_{1-\xi}  \mathcal {\tilde G}_{k} \mathcal W_{\xi}y, \qquad
x\mathcal W_{\xi} \mathcal W_{\xi}  \mathcal {\tilde G}_{k} \mathcal W_{1-\xi}y, \qquad
x\mathcal W_{\xi} \mathcal W_{\xi} \mathcal W_{1-\xi}   \mathcal {\tilde G}_{k}y.
 \label{eq:longer2}
   \end{align}
 \noindent Let $u$ denote the element on the left in \eqref{eq:longer}. Note that $u$ is the product of the terms in the sequence
 $(a_1, a_2, \ldots, a_{i-2}, \mathcal {\tilde G}_k, \mathcal W_{1-\xi}, \mathcal W_{\xi}, \mathcal W_{\xi}, a_{i+1}, \ldots, a_{N-1}, a_N)$.
  This sequence satisfies (i). This sequence has length $N+2$, so is not a CE by the maximality of $N$. By these comments $u \in U_d$.
  Similarly, each of
   the other five elements  in \eqref{eq:longer}, \eqref{eq:longer2} is contained in $U_d$.
 Consequently $x \lbrack a_{i-1}, a_i\rbrack y\in U_d$. This proves the claim under the assumption $k\geq 1$.
 Next assume that $k=0$.
 By the form of the relations (iii), (iv) in 
Definition \ref{thm:m1}, the commutator $\lbrack a_{i-1}, a_i \rbrack$ is a linear combination of
\begin{align*}
&\mathcal W_{1-\xi} \mathcal W_{\xi} \mathcal W_{\xi}, \qquad
\mathcal W_{\xi} \mathcal W_{1-\xi} \mathcal W_{\xi}, \qquad 
\mathcal W_{\xi} \mathcal W_{\xi}  \mathcal W_{1-\xi}.
   \end{align*}
   Therefore $x\lbrack a_{i-1}, a_i\rbrack y$ is a linear combination of
   \begin{align}
    \label{eq:longer3}
&x\mathcal W_{1-\xi} \mathcal W_{\xi} \mathcal W_{\xi}y, \qquad
x\mathcal W_{\xi} \mathcal W_{1-\xi} \mathcal W_{\xi}y, \qquad 
x\mathcal W_{\xi} \mathcal W_{\xi}  \mathcal W_{1-\xi}y.
   \end{align}
 \noindent Let $v$ denote the element on the left in \eqref{eq:longer3}. Note that $v$ is the product of the terms in the sequence
 $(a_1, a_2, \ldots, a_{i-2}, \mathcal W_{1-\xi}, \mathcal W_{\xi}, \mathcal W_{\xi}, a_{i+1}, \ldots, a_{N-1}, a_N)$.
  This sequence satisfies (i). This sequence has length $N+1$, so is not a CE by the maximality of $N$. By these comments $v \in U_d$.
  Similarly, each of
   the other two elements  in \eqref{eq:longer3} is contained in $U_d$.
 Consequently $x \lbrack a_{i-1}, a_i\rbrack y\in U_d$. This proves the claim under the assumption $k=0$.
 We have now proved the claim. Consider the sequence    $(a_1, a_2, \ldots, a_{i-2}, a_i, a_{i-1}, a_{i+1}, \ldots, a_{N-1}, a_N)$. This sequence satisfies (i).
 For this sequence, the product of the terms is equal to $x a_i a_{i-1} y$. We have  $x a_i a_{i-1} y \not\in U_d$, because
 $x a_{i-1}a_i y \not\in U_d$ by construction
and $x\lbrack a_{i-1},a_i \rbrack y \in U_d$ by the claim.
 By these comments, the sequence 
   $(a_1, a_2, \ldots, a_{i-2}, a_i, a_{i-1}, a_{i+1}, \ldots, a_{N-1}, a_N)$ is a CE. This CE has length $N$ and $M-1$ inversions, and this contradicts 
 the minimality of $M$. The result follows.
\end{proof}

\section{How the filtrations of $\mathcal O_q$ and $\mathcal O^\vee_q$ are related}

\noindent Recall the algebra homomorphism $\natural: \mathcal O^\vee_q \to \mathcal O_q$ from Lemma \ref{lem:step1}.
In the previous three sections, we discussed the filtration $\lbrace \mathcal O_d \rbrace_{d \in \mathbb N}$ of $\mathcal O_q$ and  the
filtration $\lbrace \mathcal O^\vee_d \rbrace_{d \in \mathbb N}$ of $\mathcal O^\vee_q$. In this section, we consider how these filtrations are related. We use
this relationship to show that $\natural$ is bijective. We use the bijectivity of $\natural$ to  prove Theorem \ref{thm:m1Int} and Theorem \ref{thm:m2}.

\begin{lemma} \label{thm:2} 
We have $\natural ( \mathcal O^\vee_d)=\mathcal O_d$ for $d \in \mathbb N$.
\end{lemma}
\begin{proof}  Our proof is by induction on $d$.
The result holds for $d=0$, since $\mathcal O^\vee_0 = \mathbb F 1$ and $\mathcal O_0 = \mathbb F 1$.
The result holds for $d=1$ by Lemmas \ref{lem:factor}(i), \ref{lem:factorV}(i).
For $d\geq 2$ we use induction along with \eqref{eq:odd}, \eqref{eq:oddV}  to obtain
\begin{align*} 
\natural(\mathcal O^\vee_d) &= 
\natural\biggl(E^\vee_d+ \mathcal O^\vee_{d-1} + \sum_{k=1}^{d-1} {\rm Span}(\mathcal O^\vee_k \mathcal O^\vee_{d-k})\biggr) \\
&=\natural(E^\vee_d)+ \natural(\mathcal O^\vee_{d-1}) + \sum_{k=1}^{d-1} {\rm Span}\bigl(\natural (\mathcal O^\vee_k)\natural( \mathcal O^\vee_{d-k})\bigr) \\
 &=E_d+ \mathcal O_{d-1} + \sum_{k=1}^{d-1} {\rm Span}(\mathcal O_k \mathcal O_{d-k}) 
 \\
&=\mathcal O_d.
\end{align*}
\end{proof}

\noindent We hope to show that $\natural$  is bijective. This result will follow from Lemma \ref{thm:2},  if we can show that ${\rm dim}(\mathcal O^\vee_d)={\rm dim}(\mathcal O_d)$ for $d \in \mathbb N$.
By Lemma \ref{thm:2}, ${\rm dim}(\mathcal O^\vee_d)\geq {\rm dim}(\mathcal O_d)$ for $d \in \mathbb N$.
To show the reverse inequality, we will display an $\mathbb F$-linear map $\mathcal O_q \to \mathcal O^\vee_q$ that sends $\mathcal O_d $ onto $\mathcal O^\vee_d$ for $d \in \mathbb N$.
This map is displayed in Proposition \ref{prop:back} below.


  
  \begin{lemma} \label{lem:vphi}
  There exists an algebra homomorphism $\flat: O_q \to \mathcal O^\vee_q$ that sends $W_0 \mapsto \mathcal W_0$ and $W_1 \mapsto \mathcal W_1$.
  \end{lemma} 
  \begin{proof} In the algebra $\mathcal O^\vee_q$ the elements  $\mathcal W_0$, $\mathcal W_1$ satisfy the $q$-Dolan/Grady relations, and these are the defining
  relations for $O_q$. 
   \end{proof} 
  

\begin{lemma}\label{lem:ztoO}
There exists an algebra homomorphism $\sharp : \mathbb F \lbrack z_1, z_2, \ldots \rbrack \to \mathcal O^\vee_q$ that sends
$z_n \mapsto \mathcal {\tilde G}_n$ for $n\geq 1$.
\end{lemma}
\begin{proof} By Definition \ref{thm:m1}(vii).
\end{proof}

\begin{lemma} \label{lem:fs}
There exists an $\mathbb F$-linear map $\varphi: O_q \otimes \mathbb F \lbrack z_1, z_2, \ldots \rbrack \to \mathcal O^\vee_q$ that sends $w \otimes z \mapsto \flat(w) \sharp(z)$ for all $w \in O_q$ and $z \in 
\mathbb F \lbrack z_1, z_2, \ldots \rbrack$.
\end{lemma}
\begin{proof} The map  $\flat \otimes \sharp: O_q \otimes \mathbb F \lbrack z_1, z_2, \ldots \rbrack \to\mathcal O^\vee_q \otimes \mathcal O^\vee_q$ sends $w \otimes z \mapsto \flat(w) \otimes \sharp(z)$ for
all $w \in O_q$ and $z \in \mathbb F \lbrack z_1, z_2, \ldots \rbrack$. The map $\flat \otimes \sharp$ is 
an algebra homomorphism, and hence $\mathbb F$-linear.
The multiplication map 
${\rm mult}: \mathcal O^\vee_q \otimes \mathcal O^\vee_q \to \mathcal O^\vee_q$ sends $u \otimes v \mapsto uv$ for all $u,v \in \mathcal O^\vee_q$. This map is $\mathbb F$-linear.
The composition
\begin{equation*}
{\begin{CD}
\varphi: \quad O_q \otimes \mathbb F \lbrack z_1, z_2, \ldots \rbrack@>>\flat \otimes \sharp > \mathcal O^\vee_q \otimes \mathcal O^\vee_q @>>{\rm mult} > \mathcal O^\vee_q
         \end{CD}}
\end{equation*}
meets the requirements of the lemma statement.
\end{proof}

\noindent Recall the algebra isomorphism $\phi: O_q \otimes \mathbb F \lbrack z_1, z_2, \ldots \rbrack \to \mathcal O_q$ from around \eqref{eq:phiAction}. In Proposition
\ref{def:cOn} 
we described $\phi^{-1} (\mathcal O_d)$ for $d \in \mathbb N$.
\begin{proposition}
\label{prop:back}
We have $\varphi\bigl( \phi^{-1} (\mathcal O_d )\bigr) = \mathcal O^\vee_d$ for $d \in \mathbb N$.
\end{proposition}
\begin{proof} Evaluate $\varphi\bigl(\phi^{-1} (\mathcal O_d )\bigr)$ using 
Proposition \ref{def:cOn},  and compare the result with $\mathcal O^\vee_d$ using
Proposition \ref{def:AnV2}.
\end{proof}

\begin{proposition} \label{lem:dd} We have ${\rm dim} (\mathcal O^\vee_d)={\rm dim}(\mathcal O_d)$ for $d \in \mathbb N$.
\end{proposition}
\begin{proof} We have  ${\rm dim} (\mathcal O^\vee_d)\geq {\rm dim}(\mathcal O_d)$ by Lemma  \ref{thm:2}, and
${\rm dim} (\mathcal O^\vee_d)\leq {\rm dim}(\mathcal O_d)$ by Proposition \ref{prop:back}. 
\end{proof}

\begin{theorem} \label{thm:nat}
The map $\natural$ is an algebra isomorphism.
\end{theorem}
\begin{proof} By Lemma \ref{lem:step1} the map $\natural$ is a surjective algebra homomorphism. To show that $\natural$ is an isomorphism, it suffices to show 
that $\natural$ is injective. Let $K$ denote the kernel of $\natural$.
For $d \in \mathbb N$ we have $\natural(\mathcal O^\vee_d)= \mathcal O_d$ by Lemma  \ref{thm:2}, so by Proposition
\ref{lem:dd}  the restriction of $\natural $ to $\mathcal O^\vee_d$ gives a bijection $\mathcal O^\vee_d \to \mathcal O_d$, so
$\mathcal O^\vee_d \cap K=0$. By this and item (iii) below Definition \ref{def:AnV}, we obtain $K=0$. Therefore $\natural$ is injective.
We have shown that $\natural$ is an algebra isomorphism.
\end{proof}

\noindent {\it Proof of Theorem \ref{thm:m1Int}}. This is a reformulation of Theorem \ref{thm:nat}. \hfill $\Box$ \\

\begin{definition}\rm By the {\it compact  presentation} of $\mathcal O_q$ we mean the presentation given in Theorem \ref{thm:m1Int}.
\end{definition}

\noindent Before proving Theorem \ref{thm:m2}, we take a closer look at the map $\varphi$ from Lemma  \ref{lem:fs}.

\begin{proposition} \label{prop:bij} The map $\varphi$ is an isomorphism of vector spaces.
\end{proposition}
\begin{proof} Adapting the proof of Theorem \ref{thm:nat} and using Propositions \ref{prop:back}, \ref{lem:dd} we find that the composition $\varphi \circ \phi^{-1}$ is an isomorphism of vector spaces. The map $\phi$ is
an algebra isomorphism and hence an isomorphism of vector spaces. The result follows by linear algebra.
\end{proof}

\noindent {\it Proof of Theorem \ref{thm:m2}.} (i) By Lemma \ref{lem:iota}.
\\ \noindent (ii), (iii).
For notational convenience, we identify $\mathcal O^\vee_q$ with $\mathcal O_q$ via the bijection $\natural$. From this point of view, we revisit a few previous results.
Lemma \ref{lem:vphi} gives 
 an algebra homomorphism $\flat: O_q \to \mathcal O_q$ that coincides with the homomorphism $\iota$ from
 Lemma \ref{lem:iota}.
 Lemma \ref{lem:ztoO} gives  an algebra homomorphism $\sharp : \mathbb F \lbrack z_1, z_2, \ldots \rbrack \to \mathcal O_q$ that sends
$z_n \mapsto \mathcal {\tilde G}_n$ for $n\geq 1$.
Lemma \ref{lem:fs} gives an $\mathbb F$-linear map $\varphi: O_q \otimes \mathbb F \lbrack z_1, z_2, \ldots \rbrack \to \mathcal O_q$ that sends $w \otimes z \mapsto \flat(w) \sharp(z)$ for all $w \in O_q$ and $z \in 
\mathbb F \lbrack z_1, z_2, \ldots \rbrack$.
The map $\varphi$ is bijective by Proposition 
\ref{prop:bij}, and this implies (ii), (iii).
\\
\noindent  (iv)  Apply the antiautomorphism $\tau$ everywhere in (iii).
 \hfill $\Box$ \\

\noindent  We finish this section with a comment. Theorem \ref{thm:m1Int} shows that the relations 
\eqref{eq:3p1}--\eqref{eq:3p11}  are redundant in the following sense.
\begin{proposition}
\label{prop:mingen} \rm
The relations \eqref{eq:3p1}--\eqref{eq:3p11} 
are implied by the relations listed in (i)--(iii) below:
\begin{enumerate}
\item[\rm (i)]  \eqref{eq:3p1}--\eqref{eq:3p3};
\item[\rm (ii)]  \eqref{eq:3p4} with $\ell=0$;
\item[\rm (iii)] the equations on the right in \eqref{eq:3p10}.
\end{enumerate}
\end{proposition}
\begin{proof} Examining the proof
Lemmas \ref{lem:newrels1}, \ref{eq:DG}, \ref{lem:newrels} we find that the
relations listed in (i)--(iii) above are the only ones used to obtain the relations (i)--(vii) in Theorem \ref{thm:m1Int}. The relations (i)--(vii) in Theorem \ref{thm:m1Int}
imply the relations \eqref{eq:3p1}--\eqref{eq:3p11}, by the meaning of Theorem \ref{thm:m1Int}. The result follows.
\end{proof}

\section{The elements $\lbrace \mathcal W_{-k}\rbrace_{k \in \mathbb N}$,  $\lbrace \mathcal W_{k+1}\rbrace_{k \in \mathbb N}$ revisited}

\noindent 
We return our attention to the elements $\lbrace \mathcal W_{-k}\rbrace_{k \in \mathbb N}$,  $\lbrace \mathcal W_{k+1}\rbrace_{k \in \mathbb N}$ in $\mathcal O_q$.
In  \eqref{eq:WmkA}--\eqref{eq:Wkp1B} we expressed these elements
in terms of the essential generators for $\mathcal O_q$. In this section, we give two more versions of
 \eqref{eq:WmkA}--\eqref{eq:Wkp1B} that are motivated by Theorem \ref{thm:m2}(iii),(iv) and also \cite[Proposition~8.4]{conj}. These versions are given in Propositions
 \ref{prop:WWalt}, \ref{prop:WWalta} below.
 Recall from Theorem \ref{thm:m2}(i) that the $q$-Onsager algebra $O_q$ is isomorphic to the subalgebra
 $\langle \mathcal W_0, \mathcal W_1 \rangle $  of $\mathcal O_q$.
 Following Baseilhac and Kolb \cite{BK} we define some elements in   $\langle \mathcal W_0, \mathcal W_1 \rangle $, denoted
\begin{align}
\lbrace B_{n \delta+ \alpha_0} \rbrace_{n=0}^\infty,
\qquad \quad 
\lbrace B_{n \delta+ \alpha_1} \rbrace_{n=0}^\infty,
\qquad \quad 
\lbrace B_{n \delta} \rbrace_{n=1}^\infty.
\label{eq:Upbw}
\end{align}
These elements are defined recursively as follows. Write $B_\delta  = q^{-2}\mathcal W_1 \mathcal W_0 - \mathcal W_0 \mathcal W_1$. We have
\begin{align}
&B_{\alpha_0}=\mathcal W_0,  \qquad \qquad 
B_{\delta+\alpha_0} = \mathcal W_1 + 
\frac{q \lbrack B_{\delta}, \mathcal W_0\rbrack}{(q-q^{-1})(q^2-q^{-2})},
\label{eq:line1}
\\
&
B_{n \delta+\alpha_0} = B_{(n-2)\delta+\alpha_0}
+ 
\frac{q \lbrack B_{\delta}, B_{(n-1)\delta+\alpha_0}\rbrack}{(q-q^{-1})(q^2-q^{-2})} \qquad \qquad n\geq 2
\label{eq:line2}
\end{align}
and 
\begin{align}
&B_{\alpha_1}=\mathcal W_1,  \qquad \qquad 
B_{\delta+\alpha_1} = \mathcal W_0 - 
\frac{q \lbrack B_{\delta}, \mathcal W_1\rbrack}{(q-q^{-1})(q^2-q^{-2})},
\label{eq:line3}
\\
&
B_{n \delta+\alpha_1} = B_{(n-2)\delta+\alpha_1}
- 
\frac{q \lbrack B_{\delta}, B_{(n-1)\delta+\alpha_1}\rbrack}{(q-q^{-1})(q^2-q^{-2})} \qquad \qquad n\geq 2.
\label{eq:line4}
\end{align}
Moreover for $n\geq 1$,
\begin{equation}
\label{eq:Bdelta}
B_{n \delta} = 
q^{-2}  B_{(n-1)\delta+\alpha_1} \mathcal W_0
- \mathcal W_0 B_{(n-1)\delta+\alpha_1}  + 
(q^{-2}-1)\sum_{\ell=0}^{n-2} B_{\ell \delta+\alpha_1}
B_{(n-\ell-2) \delta+\alpha_1}.
\end{equation}
By \cite[Proposition~5.12]{BK} the elements $\lbrace B_{n\delta}\rbrace_{n=1}^\infty$ mutually commute.
\medskip

\noindent We will not use the following fact, but we mention it for completeness.
\begin{lemma}
\label{prop:damiani} 
{\rm (See \cite[Theorem~4.5]{BK}.)}
Assume that $q$ is transcendental over $\mathbb F$. Then
 a PBW basis for 
  $\langle \mathcal W_0, \mathcal W_1 \rangle $
 is obtained by the elements {\rm \eqref{eq:Upbw}} in any linear
order.
\end{lemma}

\noindent The elements $\lbrace B_{n\delta}\rbrace_{n=1}^\infty $
are defined using the formula 
\eqref{eq:Bdelta}. We mention another formula for
$\lbrace B_{n\delta}\rbrace_{n=1}^\infty$.
According to \cite[Section~5.2]{BK} the following
holds  for $n\geq 1$:
\begin{equation}
\label{eq:Bdel2}
B_{n \delta} = 
q^{-2} \mathcal W_1 B_{(n-1)\delta+\alpha_0} 
-  B_{(n-1)\delta+\alpha_0} \mathcal W_1 + 
(q^{-2}-1)\sum_{\ell=0}^{n-2} B_{\ell \delta+\alpha_0}
B_{(n-\ell-2) \delta+\alpha_0}.
\end{equation}

\noindent 
Recall the antiautomorphism $\tau$ of $\mathcal O_q$, from Definition \ref{def:tauA}.
\begin{lemma} \label{lem:asym2} The antiautomorphism $\tau$  sends
$B_{n\delta+\alpha_0}\leftrightarrow B_{n\delta+\alpha_1}$ for $n \in \mathbb N$, and fixes $B_{n\delta}$ for $n\geq 1$.
\end{lemma}
\begin{proof} The first assertion is verified by comparing \eqref{eq:line1}, \eqref{eq:line2} with \eqref{eq:line3}, \eqref{eq:line4}. The second assertion is verified by
comparing \eqref{eq:Bdelta}, \eqref{eq:Bdel2}.
\end{proof}

\noindent The rest of this section is  motivated by \cite[Section~8]{conj}.

\begin{lemma} \label{lem:G1}
The element $\mathcal {\tilde G}_1+ q B_{\delta}$ is central in $\mathcal O_q$.
%
\end{lemma}
\begin{proof} 
By Theorem \ref{thm:m1Int}(iii),(iv) the element $\mathcal {\tilde G}_1 - \lbrack \mathcal W_0, \mathcal W_1 \rbrack_q$ commutes with $\mathcal W_0$ and $\mathcal W_1$.
 By our discussion of $\phi$ around \eqref{eq:phiAction}, the
 algebra $\mathcal O_q$ is generated by  $\mathcal W_0$, $\mathcal W_1$ together with the center of $\mathcal O_q$.
By these comments, the element $\mathcal {\tilde G}_1 - \lbrack \mathcal W_0, \mathcal W_1 \rbrack_q$ commutes with everything in $\mathcal O_q$.
By construction $\lbrack \mathcal W_0, \mathcal W_1 \rbrack_q = -q B_\delta$. The result follows.
\end{proof}

\begin{lemma} \label{lem:GWcom} For $k \in \mathbb N$ the following hold in $\mathcal O_q$;
\begin{enumerate}
\item[\rm (i)] $\lbrack \mathcal {\tilde G}_{k+1}, \mathcal W_0\rbrack_q = (q-q^{-1})\mathcal W_0 \mathcal {\tilde G}_{k+1} -q^2 \lbrack B_\delta, \mathcal W_{-k} \rbrack$;
\item[\rm (ii)] $\lbrack \mathcal W_1, \mathcal {\tilde G}_{k+1} \rbrack_q = (q-q^{-1})\mathcal W_1 \mathcal {\tilde G}_{k+1} +\lbrack B_\delta, \mathcal W_{k+1} \rbrack$.
\end{enumerate}
\end{lemma}
\begin{proof} (i) Observe that
\begin{align*}
\lbrack \mathcal {\tilde G}_{k+1},\mathcal W_0 \rbrack_q = (q-q^{-1}) \mathcal W_0 \mathcal {\tilde G}_{k+1}+q \lbrack \mathcal {\tilde G}_{k+1}, \mathcal W_0\rbrack.
\end{align*}
Setting $\ell=0$ in  \eqref{eq:3p7}  and using Lemma \ref{lem:G1},
\begin{align*}
\lbrack \mathcal {\tilde G}_{k+1}, \mathcal W_0\rbrack = \lbrack \mathcal {\tilde G}_1, \mathcal W_{-k} \rbrack = -q \lbrack B_\delta, \mathcal W_{-k} \rbrack.
\end{align*}
\noindent The result follows.
\\
\noindent (ii) Observe that
\begin{align*}
\lbrack \mathcal W_1, \mathcal {\tilde G}_{k+1} \rbrack_q = (q-q^{-1}) \mathcal W_1 \mathcal {\tilde G}_{k+1}-q^{-1} \lbrack \mathcal {\tilde G}_{k+1}, \mathcal W_1\rbrack.
\end{align*}
Setting $\ell=0$ in  \eqref{eq:3p9} and using Lemma \ref{lem:G1},
\begin{align*}
\lbrack \mathcal {\tilde G}_{k+1}, \mathcal W_1\rbrack = \lbrack \mathcal {\tilde G}_1, \mathcal W_{k+1} \rbrack = -q \lbrack B_\delta, \mathcal W_{k+1} \rbrack.
\end{align*}
\noindent The result follows.
\end{proof}

\begin{lemma}\label{lem:Wind}
For $n\geq 1$ the following hold in $\mathcal O_q$:
\begin{align}
\label{eq:Wind1}
\mathcal W_{-n} &= \mathcal W_n -\frac{(q-q^{-1}) \mathcal W_0 \mathcal{\tilde G}_n}{(q^2-q^{-2})^2} + \frac{q^2 \lbrack B_\delta, \mathcal W_{1-n}\rbrack}{(q^2-q^{-2})^2};
\\
\mathcal W_{n+1} &=\mathcal W_{1-n}-\frac{(q-q^{-1}) \mathcal W_1\mathcal {\tilde G}_n}{(q^2-q^{-2})^2} -\frac{\lbrack B_\delta, \mathcal W_n\rbrack}{(q^2-q^{-2})^2}.
\label{eq:Wind2}
\end{align}
\end{lemma}
\begin{proof} Use the equations on the right in
\eqref{eq:3p2}, \eqref{eq:3p3} along with Lemma  \ref{lem:GWcom}.
\end{proof}
\noindent 
We recall some notation from \cite{BK}. For a negative integer $k$ define
\begin{align*}
B_{k \delta + \alpha_0} = B_{(-k-1)\delta + \alpha_1}, \qquad \qquad 
B_{k\delta+\alpha_1} = B_{(-k-1)\delta+\alpha_0}.
\end{align*}
We have
\begin{align}
 B_{r\delta+\alpha_0} = B_{s\delta+\alpha_1} \qquad \qquad (r,s \in \mathbb Z, \quad r+s=-1).
 \label{eq:extend}
 \end{align}

\begin{lemma} For $n\in \mathbb Z$ the following hold in $\mathcal O_q$:
\begin{align}
\label{eq:nnot1}
\frac{q\lbrack B_\delta, B_{n\delta+\alpha_0} \rbrack}{(q-q^{-1})(q^2-q^{-2})}&= B_{(n+1) \delta +\alpha_0} - B_{(n-1)\delta +\alpha_0};
\\
\label{eq:nnot2}
\frac{q\lbrack B_\delta, B_{n\delta+\alpha_1} \rbrack}{(q-q^{-1})(q^2-q^{-2})}&= B_{(n-1)\delta +\alpha_1} - B_{(n+1)\delta +\alpha_1}.
\end{align}
\end{lemma}
\begin{proof} Use \eqref{eq:line1}--\eqref{eq:line4} and \eqref{eq:extend}.
\end{proof}

\begin{proposition}
\label{prop:WWalt} For $n \in \mathbb N$ the following hold in $\mathcal O_q$:
\begin{align}
\label{eq:recWnm1}
\mathcal W_{-n} &=-(q-q^{-1})^{-1} \sum_{k=0}^n \sum_{\ell=0}^k \binom{k}{\ell} q^{k-2 \ell} \lbrack 2 \rbrack^{-k-2}_q B_{(k-2\ell)\delta + \alpha_0} \mathcal {\tilde G}_{n-k};
\\
\mathcal W_{n+1} &=  -(q-q^{-1})^{-1}  \sum_{k=0}^n \sum_{\ell=0}^k \binom{k}{\ell} q^{2 \ell-k} \lbrack 2 \rbrack^{-k-2}_q B_{(k-2\ell)\delta + \alpha_1} \mathcal {\tilde G}_{n-k}.
\label{eq:recWnm2}
\end{align}
\end{proposition}
\begin{proof} We use induction on $n$.
First assume that $n=0$. Then \eqref{eq:recWnm1}, \eqref{eq:recWnm2} hold. Next assume that $n\geq 1$.
To obtain \eqref{eq:recWnm1}, evaluate the right-hand side of 
 \eqref{eq:Wind1} 
 using induction along with \eqref{eq:extend}, \eqref{eq:nnot1}.
 To obtain \eqref{eq:recWnm2}, evaluate the right-hand side of 
 \eqref{eq:Wind2}  
  using induction along with \eqref{eq:extend}, \eqref{eq:nnot2}.
\end{proof}

\begin{proposition}
\label{prop:WWalta} For $n \in \mathbb N$ the following hold in $\mathcal O_q$:
\begin{align}
\mathcal W_{-n} &=  -(q-q^{-1})^{-1}  \sum_{k=0}^n \sum_{\ell=0}^k \binom{k}{\ell} q^{2 \ell-k} \lbrack 2 \rbrack^{-k-2}_q 
 \mathcal {\tilde G}_{n-k} B_{(k-2\ell)\delta + \alpha_0};
\label{eq:recWnm1a}
\\
\mathcal W_{n+1} &=-(q-q^{-1})^{-1} \sum_{k=0}^n \sum_{\ell=0}^k \binom{k}{\ell} q^{k-2 \ell} \lbrack 2 \rbrack^{-k-2}_q  \mathcal {\tilde G}_{n-k} B_{(k-2\ell)\delta + \alpha_1}.
\label{eq:recWnm2a}
\end{align}
\end{proposition}
\begin{proof} To get \eqref{eq:recWnm1a} (resp. \eqref{eq:recWnm2a}), apply the antiautomorphism $\tau$ to each side of \eqref{eq:recWnm2} (resp. \eqref{eq:recWnm1})  and evaluate the
results using Definition \ref{def:tauA}
and
Lemma  \ref{lem:asym2}. 
\end{proof}

\section{A variation on the main results}

\noindent Recall the automorphism $\sigma$ of $\mathcal O_q$, from Lemma \ref{lem:autA}.
In this section we use $\sigma$ to obtain a variation on Theorems \ref{thm:m1Int}, \ref{thm:m2} and Propositions \ref{prop:WWalt}, \ref{prop:WWalta}.
\medskip

\noindent
Here is a variation on Theorem \ref{thm:m1Int}.
 \begin{proposition} \label{thm:m1com}
 The algebra $\mathcal O_q$  has a presentation by generators 
$\mathcal W_0$, $\mathcal W_1$, $\lbrace \mathcal {G}_{k+1} \rbrace_{k \in \mathbb N}$ and relations
\begin{enumerate}
\item[\rm (i)]
$\lbrack \mathcal W_0, \lbrack \mathcal W_0, \lbrack \mathcal W_0, \mathcal W_1\rbrack_q \rbrack_{q^{-1}} \rbrack =(q^2-q^{-2})^2 \lbrack \mathcal W_1, \mathcal W_0 \rbrack$;
\item[\rm (ii)]
$\lbrack \mathcal W_1, \lbrack \mathcal W_1, \lbrack \mathcal W_1, \mathcal W_0\rbrack_q \rbrack_{q^{-1}}\rbrack = (q^2-q^{-2})^2  \lbrack \mathcal W_0, \mathcal W_1 \rbrack$;
\item[\rm (iii)] 
$\lbrack \mathcal W_1, \mathcal {G}_1 \rbrack = 
\lbrack \mathcal W_1, \lbrack \mathcal W_1, \mathcal W_0 \rbrack_q \rbrack$;
\item[\rm (iv)] $\lbrack \mathcal { G}_1, \mathcal W_0 \rbrack = 
\lbrack \lbrack \mathcal W_1, \mathcal W_0 \rbrack_q, \mathcal W_0 \rbrack$;
\item[\rm (v)] 
for $k\geq 1$,
\begin{align*}
&\lbrack \mathcal {G}_{k+1}, \mathcal W_1 \rbrack = 
\frac{
\lbrack \mathcal W_1, \lbrack \mathcal W_1, \lbrack \mathcal W_0,
\mathcal {G}_k
\rbrack_q
\rbrack_q
\rbrack}{(q^2-q^{-2})^2};
\end{align*}
\item[\rm (vi)] for $k\geq 1$,
\begin{align*}
\lbrack \mathcal W_0, \mathcal {G}_{k+1}\rbrack = 
\frac{
\lbrack\lbrack \lbrack \mathcal {G}_k, \mathcal W_1 \rbrack_q, 
\mathcal W_0 \rbrack_q,
\mathcal W_0
\rbrack}{(q^2-q^{-2})^2};
\end{align*}
\item[\rm (vii)] for $k, \ell \in \mathbb N$,
\begin{align*}
\lbrack \mathcal {G}_{k+1}, \mathcal {G}_{\ell+1} \rbrack=0.
\end{align*}
\end{enumerate}
\end{proposition}
\begin{proof} Apply the automorphism $\sigma$ everywhere in Theorem \ref{thm:m1Int}, and use Lemma \ref{lem:autA}.
\end{proof}

\noindent Next we give a variation on Theorem \ref{thm:m2}. Let $\mathcal G$ denote the subalgebra of $\mathcal O_q$ generated by $\lbrace \mathcal G_{k+1} \rbrace_{k\in \mathbb N}$.
\begin{proposition}
\label{prop:m2Com}
Theorem \ref{thm:m2} remains valid if $\mathcal {\tilde G}$ is replaced by $\mathcal G$.
\end{proposition}
\begin{proof} Apply the automorphism $\sigma$ everywhere in Theorem \ref{thm:m2}, and use Lemma \ref{lem:autA}.
\end{proof}
\noindent Next we give a variation on Propositions  \ref{prop:WWalt}, \ref{prop:WWalta}.
For $n \in \mathbb N$ let  ${\tilde B}_{n \delta+ \alpha_0} $ (resp. ${\tilde B}_{n \delta+ \alpha_1} $) denote the $\sigma$-image of $ B_{n \delta+ \alpha_0} $ (resp. $ B_{n \delta+ \alpha_1} $).
For $n\geq 1$ let ${\tilde B}_{n\delta}$ denote the $\sigma$-image of $B_{n\delta}$. These elements satisfy the following recursion.
We have $\tilde B_\delta  = q^{-2}\mathcal W_0 \mathcal W_1 - \mathcal W_1 \mathcal W_0$. We have
\begin{align*}
&\tilde B_{\alpha_0}=\mathcal W_1,  \qquad \qquad 
\tilde B_{\delta+\alpha_0} = \mathcal W_0 + 
\frac{q \lbrack \tilde B_{\delta}, \mathcal W_1\rbrack}{(q-q^{-1})(q^2-q^{-2})},
\\
&
\tilde B_{n \delta+\alpha_0} = \tilde B_{(n-2)\delta+\alpha_0}
+ 
\frac{q \lbrack \tilde B_{\delta}, \tilde B_{(n-1)\delta+\alpha_0}\rbrack}{(q-q^{-1})(q^2-q^{-2})} \qquad \qquad n\geq 2
\end{align*}
and 
\begin{align*}
&\tilde B_{\alpha_1}=\mathcal W_0,  \qquad \qquad 
\tilde B_{\delta+\alpha_1} = \mathcal W_1 - 
\frac{q \lbrack {\tilde B}_{\delta}, \mathcal W_0 \rbrack}{(q-q^{-1})(q^2-q^{-2})},
\\
&
\tilde B_{n \delta+\alpha_1} = \tilde B_{(n-2)\delta+\alpha_1}
- 
\frac{q \lbrack \tilde B_{\delta}, \tilde B_{(n-1)\delta+\alpha_1}\rbrack}{(q-q^{-1})(q^2-q^{-2})} \qquad \qquad n\geq 2.
\end{align*}
Moreover for $n\geq 1$,
\begin{align*}
\tilde B_{n \delta} &= 
q^{-2}  \tilde B_{(n-1)\delta+\alpha_1} \mathcal W_1
- \mathcal W_1 \tilde B_{(n-1)\delta+\alpha_1}  + 
(q^{-2}-1)\sum_{\ell=0}^{n-2} \tilde B_{\ell \delta+\alpha_1}
\tilde B_{(n-\ell-2) \delta+\alpha_1}
\\
&=q^{-2} \mathcal W_0 \tilde B_{(n-1)\delta+\alpha_0} 
-  \tilde B_{(n-1)\delta+\alpha_0} \mathcal W_0 + 
(q^{-2}-1)\sum_{\ell=0}^{n-2} \tilde B_{\ell \delta+\alpha_0}
\tilde B_{(n-\ell-2) \delta+\alpha_0}.
\end{align*}

\noindent  We emphasize that the elements 
\begin{align*}
\lbrace {\tilde B}_{n \delta+ \alpha_0} \rbrace_{n=0}^\infty,
\qquad \quad 
\lbrace {\tilde B}_{n \delta+ \alpha_1} \rbrace_{n=0}^\infty,
\qquad \quad 
\lbrace {\tilde B}_{n \delta} \rbrace_{n=1}^\infty
\end{align*}
are contained in $\langle \mathcal W_0, \mathcal W_1 \rangle$.

\begin{proposition}
\label{prop:WWaltv} For $n \in \mathbb N$ the following hold in $\mathcal O_q$:
\begin{align*}
\mathcal W_{n+1} &=-(q-q^{-1})^{-1} \sum_{k=0}^n \sum_{\ell=0}^k \binom{k}{\ell} q^{k-2 \ell} \lbrack 2 \rbrack^{-k-2}_q \tilde B_{(k-2\ell)\delta + \alpha_0} \mathcal  G_{n-k};
\\
\mathcal W_{-n} &=  -(q-q^{-1})^{-1}  \sum_{k=0}^n \sum_{\ell=0}^k \binom{k}{\ell} q^{2 \ell-k} \lbrack 2 \rbrack^{-k-2}_q \tilde B_{(k-2\ell)\delta + \alpha_1} \mathcal  G_{n-k}.
\end{align*}
\end{proposition}
\begin{proof} Apply $\sigma$ everywhere in Proposition \ref{prop:WWalt}, and use Lemma \ref{lem:autA}.
\end{proof}

\begin{proposition}
\label{prop:WWaltav} For $n \in \mathbb N$ the following hold in $\mathcal O_q$:
\begin{align*}
\mathcal W_{n+1} &=  -(q-q^{-1})^{-1}  \sum_{k=0}^n \sum_{\ell=0}^k \binom{k}{\ell} q^{2 \ell-k} \lbrack 2 \rbrack^{-k-2}_q 
 \mathcal {G}_{n-k} \tilde B_{(k-2\ell)\delta + \alpha_0};
\\
\mathcal W_{-n} &=-(q-q^{-1})^{-1} \sum_{k=0}^n \sum_{\ell=0}^k \binom{k}{\ell} q^{k-2 \ell} \lbrack 2 \rbrack^{-k-2}_q  \mathcal  G_{n-k} \tilde B_{(k-2\ell)\delta + \alpha_1}.
\end{align*}
\end{proposition}
\begin{proof} Apply $\sigma$ everywhere in Proposition \ref{prop:WWalta}, and use Lemma \ref{lem:autA}.
\end{proof}

\section{Suggestions for future research}

\noindent In this section we give some suggestions for future research. Recall the $q$-Onsager algebra $O_q$
and its alternating central extension $\mathcal O_q$.
\medskip

\begin{problem}\label{prob4}\rm
 Display an  algebra isomorphism from  $\mathcal O_q$ to a coideal subalgebra of  $U_q(\widehat{\mathfrak{gl}}_2)$.
\end{problem}

\begin{problem}\label{prob:one} \rm Assume that $\mathbb F$ is algebraically closed, and
let $V$ denote a finite-dimensional irreducible $\mathcal O_q$-module on which
$\mathcal W_0$, $\mathcal W_1$ are diagonalizable. It follows from \cite[Theorem~3.10]{qSerre} that $\mathcal W_0$, $\mathcal W_1$ act on $V$ as a tridiagonal pair.
This means that:
\begin{enumerate}
\item[\rm (i)] there exists an ordering $\lbrace V_i \rbrace_{i=0}^d$ of the eigenspaces of $\mathcal W_0$ on $V$ such that
\begin{align*}
\mathcal W_1 V_i \subseteq  V_{i-1} +V_i + V_{i+1} \qquad \quad (0 \leq i \leq d),
\end{align*}
\noindent where $V_{-1}=0$ and $V_{d+1}=0$;
\item[\rm (ii)] there exists an ordering $\lbrace V'_i \rbrace_{i=0}^d$ of the eigenspaces of $\mathcal W_1$ on $V$ such that
\begin{align*}
\mathcal W_0 V'_i \subseteq V'_{i-1} +V'_i + V'_{i+1} \qquad \quad (0 \leq i \leq d),
\end{align*}
\noindent where $V'_{-1}=0$ and $V'_{d+1}=0$.
\end{enumerate}
Show that for $k \in \mathbb N$ and $0 \leq i \leq d$,
\begin{align*}
\mathcal W_{-k} V_i &\subseteq  V_i,  \qquad \qquad \qquad \qquad \qquad
\mathcal W_{-k} V'_i \subseteq  V'_{i-1} +V'_i + V'_{i+1}, \\
\mathcal W_{k+1} V_i &\subseteq  V_{i-1} +V_i + V_{i+1},  \qquad \quad \;\;
\mathcal W_{k+1} V'_i \subseteq  V'_i, \\
\mathcal G_{k+1} V_i &\subseteq  V_{i-1} +V_i + V_{i+1},\qquad \qquad 
\mathcal G_{k+1} V'_i \subseteq  V'_{i-1} +V'_i + V'_{i+1}, \\
\mathcal {\tilde G}_{k+1} V_i &\subseteq  V_{i-1} +V_i + V_{i+1}, \qquad \qquad 
\mathcal {\tilde G}_{k+1} V'_i \subseteq  V'_{i-1} +V'_i + V'_{i+1}.
\end{align*}
\noindent See    \cite{basXXZ, BK05} along with \cite[Line~(14)]{bas4} and
\cite[Section~2.3]{basKoi} 
for explicit examples and relevant information.
\end{problem}

\begin{problem}\rm With reference to Problem \ref{prob:one}, assume that  $V_i$ and $V'_i$ have dimension one for $0 \leq i \leq d$.
Show that for $k \in \mathbb N$ the map
$\mathcal W_{-k} $ (resp. $\mathcal W_{k+1} $)
(resp. $\mathcal G_{k+1} $)
(resp. $\mathcal {\tilde G}_{k+1} $) 
acts on $V$ as a linear combination of $\mathcal W_0$ and $I$
(resp. $\mathcal W_1$ and $I$)
(resp. $\lbrack \mathcal W_1, \mathcal W_0\rbrack_q$ and $I$)
(resp. $\lbrack \mathcal W_0, \mathcal W_1\rbrack_q $ and $I$).
See \cite[Section~2.1.1]{BK05} and \cite[Section~7]{conj} for partial results.
\end{problem}

\begin{problem}\label{prob:six}\rm
 Assume that $\mathbb F$ is algebraically closed, and let $V$ denote a finite-dimensional irreducible $\mathcal O_q$-module on which
$\mathcal W_0$, $\mathcal W_1$ are diagonalizable.
Examples show that the alternating $\mathcal O_q$-generators do not necessarily act on $V$ in a diagonalizable way. Investigate
the Jordan canonical form for these actions. How does the Bockting double lowering operator 
\cite{bockting, bocktingQexp}
act on the Jordan blocks?
\end{problem}

\noindent To motivate the next problem, recall from 
\eqref{eq:3p4},
\eqref{eq:3p10}
that each of the following (i)--(iv) mutually commute:
(i) $\lbrace \mathcal W_{-k}\rbrace_{k\in \mathbb N}$;
(ii) $\lbrace \mathcal W_{k+1}\rbrace_{k\in \mathbb N}$;
(iii) $\lbrace \mathcal G_{k+1}\rbrace_{k\in \mathbb N}$;
(iv) $\lbrace \mathcal {\tilde G}_{k+1}\rbrace_{k\in \mathbb N}$.
\begin{problem}\label{prob:two} \rm Assume that $\mathbb F$ is algebraically closed, and let $V$ denote a finite-dimensional irreducible $\mathcal O_q$-module.
 Investigate the common eigenvectors for the actions of
$\lbrace \mathcal W_{-k}\rbrace_{k\in \mathbb N}$ 
(resp. $\lbrace \mathcal W_{k+1}\rbrace_{k\in \mathbb N}$)
(resp. $\lbrace \mathcal G_{k+1}\rbrace_{k\in \mathbb N}$)
(resp. $\lbrace \mathcal {\tilde G}_{k+1}\rbrace_{k\in \mathbb N}$)
on $V$. 
\end{problem}

\begin{problem} \label{prob3} \rm
Describe the subalgebra of $\mathcal O_q$ consisting of the elements that commute with $\mathcal W_0$.
\end{problem}

\begin{problem} \label{prob3a} \rm
Describe the subalgebra of $\mathcal O_q$ consisting of the elements that commute with $\lbrack \mathcal W_0, \mathcal W_1\rbrack_q$.
\end{problem}

\begin{problem}\rm Describe the subalgebra of $\mathcal O_q$ generated by $\lbrace \mathcal W_{-k}\rbrace_{k\in \mathbb N}$ and  $\lbrace \mathcal W_{k+1}\rbrace_{k\in \mathbb N}$.
\end{problem}

\begin{problem}\rm Describe the subalgebra of $\mathcal O_q$ generated by  $\lbrace \mathcal G_{k+1}\rbrace_{k\in \mathbb N}$ and  $\lbrace \mathcal {\tilde G}_{k+1}\rbrace_{k\in \mathbb N}$.
\end{problem}

\section{Acknowledgements}
The author thanks Pascal Baseilhac and Travis Scrimshaw for helpful conversations about the $q$-Onsager algebra and its alternating central extension.


%

\bigskip

\noindent Paul Terwilliger \hfil\break
\noindent Department of Mathematics \hfil\break
\noindent University of Wisconsin \hfil\break
\noindent 480 Lincoln Drive \hfil\break
\noindent Madison, WI 53706-1388 USA \hfil\break
\noindent email: {\tt terwilli@math.wisc.edu }\hfil\break

\end{document}